\documentclass{amsart}
\usepackage{amsmath}%
\usepackage{amsthm}
\usepackage{amsfonts}%
\usepackage{amssymb}%
\usepackage{graphicx}
\usepackage{bm}
\usepackage{comment}
\usepackage{mathtools}
\usepackage{bbm}

\usepackage[all,cmtip]{xy}
\usepackage{tikz-cd}
\usetikzlibrary{cd}

\usepackage[utf8]{inputenc}
\usepackage{a4}
\usepackage[hidelinks]{hyperref}
\usepackage{enumitem}

\usepackage{mathrsfs}

\DeclareFontFamily{OT1}{pzc}{}
\DeclareFontShape{OT1}{pzc}{m}{it}{<-> s * [1.0] pzcmi7t}{}
\DeclareMathAlphabet{\mathpzc}{OT1}{pzc}{m}{it}

\usepackage[singlespacing]{setspace}
\newtheorem{theorem}{Theorem}[section]

\newtheorem{corollary}[theorem]{Corollary}

\newtheorem{definition}[theorem]{Definition}

\newtheorem{lemma}[theorem]{Lemma}

\newtheorem{proposition}[theorem]{Proposition}
\newtheorem{remark}[theorem]{Remark}

\numberwithin{equation}{section}

\def\XXint#1#2#3{{\setbox0=\hbox{$#1{#2#3}{\int}$}
\vcenter{\hbox{$#2#3$}}\kern-.5\wd0}}

\newcommand{\cF}{\mathcal{F}}
\newcommand{\cI}{\mathcal{I}}
\newcommand{\R}{\mathbb{R}}

\newcommand{\N}{\mathbb{N}}
\newcommand{\Z}{\mathbb{Z}}

\newcommand{\Ha}{\mathcal{H}}
\newcommand{\leb}{\mathcal{L}}

\newcommand{\sgn}{\operatorname{sgn}}
\newcommand{\spt}{\operatorname{spt}}

\newcommand{\diam}{\operatorname{diam}}
\newcommand{\inte}{\operatorname{int}}

\newcommand{\Lip}{\operatorname{Lip}}
\newcommand{\LIP}{\operatorname{LIP}}
\newcommand{\im}{\operatorname{im}}

\newcommand{\id}{\mathrm {id}}

\renewcommand{\epsilon}{\varepsilon}

\newcommand{\on}{\:\mbox{\rule{0.1ex}{1.2ex}\rule{1.1ex}{0.1ex}}\:}

\DeclareMathOperator{\bI}{\mathbf{I}}
\DeclareMathOperator{\st}{\textup{st}}
\DeclareMathOperator{\bM}{\mathbf{M}}
\DeclareMathOperator{\bN}{\mathbf{N}}

\DeclareMathOperator{\mass}{\mathbf{M}}

\newcommand{\bb}[1]{\llbracket #1\rrbracket}
\usepackage{stmaryrd}
\newcommand{\norm}[1]{\lVert#1\rVert}

\newcommand{\mres}{\mathbin{\vrule height 1.6ex depth 0pt width
0.13ex\vrule height 0.13ex depth 0pt width 1.3ex}}

\usepackage[colorinlistoftodos]{todonotes}
\usepackage{import}
\usepackage{xifthen}
\usepackage{pdfpages}
\usepackage{transparent}
\usepackage{xcolor}
\usepackage{stmaryrd}

\makeatletter
\newcommand*{\cone}{%
	{%
		\mathpalette\@coneOf{\times}%
	}%
}
\newcommand*{\@coneOf}[2]{%
	\sbox0{$\m@th#1\mathsf{#2}$}%
	\mathsf{#2}%
	\kern-\wd0 %
	\mkern2.00mu\relax
	\nonscript\mkern0.25mu\relax
	\mathsf{#2}%
}

\usepackage{etoolbox}

\makeatletter
\patchcmd{\@setaddresses}{\indent}{\noindent}{}{}
\patchcmd{\@setaddresses}{\indent}{\noindent}{}{}
\patchcmd{\@setaddresses}{\indent}{\noindent}{}{}
\patchcmd{\@setaddresses}{\indent}{\noindent}{}{}
\makeatother

\thanks{D.M. was supported by Swiss National Science Foundation grant 212867.}

\author{Denis Marti}

\address{Department of Mathematics\\ University of Fribourg\\  Chemin du Mus\'ee 23\\  1700 Fribourg, Switzerland}
\email{denis.marti@unifr.ch}

\title[The metric fundamental class of non-orientable manifolds]{The metric fundamental class of non-orientable manifolds and manifolds with boundary}

\begin{document}

\begin{abstract}
    We introduce the metric fundamental class for metric spaces that are homeomorphic to compact, non-orientable, smooth manifolds with (possibly empty) boundary. This is an integer rectifiable current that provides an analytic representation of the topological fundamental class of the space. Under certain weak geometric conditions, we show the existence of such a current, extending earlier results for orientable, closed manifolds obtained in collaboration with Basso and Wenger. As an application, we present new rectifiability results.
\end{abstract}

\maketitle


\section{Introduction}
    In an earlier work jointly with Basso and Wenger, the author showed that, under certain geometric conditions, a metric space that is homeomorphic to a closed, orientable smooth manifold admits a non-trivial integral cycle in the sense of Ambrosio–Kirchheim \cite{ambrosio-kirchheim-2000}. This integral cycle serves as a metric analogue of the topological fundamental class of a smooth manifold and is referred to as the metric fundamental class of the manifold.

    \begin{definition}\label{def: metric fundamental class}
        Let $X$ be a metric space with finite Hausdorff $n$-measure that is homeomorphic to a closed, oriented, Riemannian $n$-manifold $M$. A metric fundamental class of $X$ is an integral current $T\in \bI_n(X)$ with $\partial T=0$ such that
        \begin{itemize}        
            \item[(a)] $\varphi_\# T = \deg(\varphi) \cdot \bb{M}$ for every Lipschitz map $\varphi \colon X \to M$;
            \item[(b)] there exists $C>0$ such that $\norm{T}\leq C \cdot \Ha^n$.
        \end{itemize}
    \end{definition}

    Here, $M$ is equipped with any Riemannian metric and $\bb{M}$ denotes its fundamental class given by integration of differential $n$-forms, and $\norm{T}$ is the mass measure of $T$. Informally speaking, the conditions (a) and (b) ensure that $T$ is compatible with the topological and metric structure of $X$, respectively. We note that a metric fundamental class, when it exists, is unique and generates the $n$th homology group via integral currents; c.f. \cite[Proposition 5.5]{BMW}. A metric space homeomorphic to a smooth manifold is called  a metric manifold. Such spaces and the relationship between their geometric and analytic properties play a central role in metric geometry; see e.g. \cite{bonk-kleiner, david-american-2016,Heinonen-Rickman-2002, meier2025energy, semmes-curves}. In this vein, the metric fundamental class offers a powerful tool for studying metric manifolds. For instance, a deep theorem of Semmes \cite{semmes-curves} guarantees the validity of a Poincar\'e inequality in Ahlfors regular and linearly locally contractible metric spaces that are homeomorphic to a closed, oriented, smooth manifold. In \cite{BMW}, a new proof of this result was provided using the metric fundamental class; see also \cite{BMW, heinonen-keith-2011, heinonen-sullivan-2002, marti2024characterization} for more applications.

    \medskip

    The goal of this article is to extend these ideas and results to manifolds that are non-orientable and to manifolds with boundary. In the non-orientable case, every top-dimensional integral current necessarily has non-zero boundary and the notion of degree is not well-defined. Therefore, we adapt the definition of a metric fundamental class. We require the current to have no boundary modulo $2$ as defined in \cite{Ambrosio-Katz, Ambrosio-Wenger}. Moreover, in Definition \ref{def: metric fundamental class}(a) we use the degree defined via singular homology with coefficients in $\Z/2\Z$. More precisely, an integer rectifiable current $T$ satisfies $\partial T = 0 \mod 2$ if $\cF_2(\partial T) = 0$. Here, $\cF_2(\partial T)$ denotes the flat norm modulo $2$ of $\partial T$. This distance is defined as the infimum of $\bM(U)+\bM(V)$ over integer rectifiable currents $U$ and $V$ of dimensions $n-1$ and $n$, respectively, such that $\partial T = U + \partial V + 2Q$ for some flat current $Q$. We refer to Section \ref{sec: flat-currents} for a detailed overview of flat currents modulo $2$. There also exists a well-developed theory of flat chains with coefficients in any normed abelian group \cite{De-pauw-rect-flat-chains} that generalizes the flat chains considered in \cite{brian-white-g-chains} to metric spaces. This theory also provides a natural framework to define a metric fundamental class for non-orientable manifolds. Nevertheless, we have chosen to work with Ambrosio–Kirchheim currents because of a direct connection with the results in \cite{BMW} for orientable manifolds.

    \subsection{Statement of main results}
        A metric space $X$ is called a metric $n$-manifold if it is homeomorphic to a compact smooth $n$-manifold $M$. We say that $X$ is orientable, non-orientable, or closed, depending on whether $M$ has the corresponding properties. In case $M$ has non-empty boundary $\partial M$, we denote by $\partial X$ the image of $\partial M$ under the homeomorphism from $X$ to $M$. Every manifold considered in this article is assumed to be connected. We write $\cI_n(X)$ for the space of integer rectifiable currents in $X$; see Section \ref{sec: currents} for the relevant definitions. 

        \medskip

        The first result provides the existence of a metric fundamental class in orientable metric manifolds that are linearly locally contractible and extends \cite[Theorem 1.1]{BMW} to manifolds with boundary. A metric space is said to be linearly locally contractible if there exists $\lambda>0$ such that every ball of radius  $0<r<\diam X/ \lambda$ is contractible within the ball with the same center and radius $\lambda r$. This property often appears in various contexts in metric geometry; see e.g. \cite{bonk-kleiner, david-american-2016, Grove-Petersen-Wu-1990,Hajlasz-Koskela-1995,Kleiner-Lang-2020,Semmes-good-metric-spaces-1996}.
    
        \begin{theorem}\label{thme: existence-orient-llc}
            Let $X$ be a compact, orientable metric $n$-manifold with finite Hausdorff $n$-measure and (possibly empty) boundary. Suppose that $X$ is linearly locally contractible and $\Ha^n(\partial X) = 0$. Then, $X$ has a metric fundamental class $T \in \cI_n(X)$ with
            \begin{equation}\label{eq: mass-measure-hausdorff-existence-weak-llc}
                C^{-1} \Ha^n \leq \norm{T} \leq C\Ha^n,
            \end{equation}
            and whenever $S \in \bI_n(X)$ satisfies $\spt (\partial S) \subset \partial X$, then there exists $k \in \Z$ such that $S = k \cdot T$. 
        \end{theorem}
        
        Here, $C$ depends only on $n$. If $X$ is closed, then $T$ is as in Definition \ref{def: metric fundamental class} and the second property holds for all $S\in \bI_n(X)$ with $\partial S= 0$. An integer rectifiable current $T \in \cI_n(X)$ in a metric manifold $X$ with non-empty boundary is said to be a metric fundamental class of $X$ if $T$ satisfies Definition \ref{def: metric fundamental class}(b) and Definition \ref{def: metric fundamental class}(a) for every Lipschitz map $\varphi\colon X \to M$ satisfying $\varphi(\partial X) \subset \partial (M)$.


\medskip


        Next, we consider non-orientable manifolds. As mentioned earlier, the definition of the metric fundamental class needs to be adapted for these manifolds.

        \begin{definition}\label{def: metric fundamental class non-oriented}
            Let $X$ be a metric space with finite Hausdorff $n$-measure that is homeomorphic to a closed, non-orientable, smooth $n$-manifold $M$. A metric fundamental class modulo $2$ of $X$ is an integer rectifiable current $T\in \cI_n(X)$ with $\partial T=0 \mod 2$ such that
            \begin{equation}\label{eq: def-metric-fundamental-class-modulo 2}
                \varphi_\# T = \deg(\varphi,\Z/2\Z) \cdot \bb{M} \textup{ for every Lipschitz map }\varphi \colon X \to M.
            \end{equation}
        \end{definition}

        If $X$ has boundary, then \eqref{eq: def-metric-fundamental-class-modulo 2} is only true for Lipschitz maps $\varphi\colon X \to M$ satisfying $\varphi(\partial X) \subset \partial M$. The manifold $M$ can be equipped with any Riemannian metric and $\bb{M}$ denotes the corresponding fundamental class modulo $2$. Unlike in the orientable case, we do not require the second property since we always have $\norm{T}_2\leq C\Ha^n$, where $C$ depends only on $n$. The degree $\deg(\varphi,\Z/2\Z)$ is defined in terms of singular homology groups with coefficients in $\Z/2\Z$ and is interpreted as either $0$ or $1$, depending on whether it is trivial or not.

\medskip
        
        The next result is the equivalent of Theorem \ref{thme: existence-orient-llc} for non-orientable manifolds.
    
        \begin{theorem}\label{thme: existence-non-orient-llc}
           Let $X$ be a compact, non-orientable metric $n$-manifold with finite Hausdorff $n$-measure and (possibly empty) boundary. Suppose that $X$ is linearly locally contractible and $\Ha^n(\partial X) = 0$. Then, $X$ has a metric fundamental class $T \in \cI_n(X)$ modulo $2$ satisfying 
            \begin{equation}\label{eq: mass-measure-hausdorff-existence-weak-llc-non-oriented}
                C^{-1} \Ha^n \leq \norm{T}_2 = \norm{T} \leq C\Ha^n,
            \end{equation}
            and whenever $S \in \bI_n(X)$ satisfies $\spt_2 (\partial S) \subset \partial X$, then $S = k \cdot T \mod 2$ for $k$ equal to either $1$ or $0$. 
        \end{theorem}

        Here, $C$ depends only on $n$, and $\norm{T}_2$ is the $2$-mass measure of $T$. The $2$-support $\spt_2 S$ of an integral current $S \in \bI_n(X)$ is defined as the support of its $2$-mass measure $\norm{S}_2$. If $X$ is closed, then the uniqueness property holds for all $S\in \cI_n(X)$ with $\partial S = 0 \mod 2$.
        
        \medskip
        
        We note that we prove Theorem \ref{thme: existence-orient-llc} and Theorem \ref{thme: existence-non-orient-llc} using a weaker version of linear local contractibility, as in the standard definition explained above. Concretely, both results (and their corollary below) hold for metric manifolds that are almost everywhere linearly locally contractible; see Section \ref{sec: llc}. As the name indicates, this can be thought of as a measure theoretic version of the standard definition. We obtain the following corollary. 
        
        \begin{corollary}\label{cor: metric-mfld-rect}
            Let $X$ be a compact metric $n$-manifold with finite Hausdorff $n$-measure and (possibly empty) boundary. Suppose that $X$ is linearly locally contractible and satisfies $\Ha^n(\partial X)=0$. Then $X$ is $n$-rectifiable.
        \end{corollary}

        This extends the rectifiability result in \cite{BMW} and is a direct consequence of the mass measure bounds of the fundamental class in Theorem \ref{thme: existence-orient-llc} and Theorem \ref{thme: existence-non-orient-llc}. For non-orientable manifolds, we are also able to prove the existence of a metric fundamental class modulo $2$ when the space is not linearly locally contractible.

        \begin{theorem}\label{thme: existence-non-orient-nagata}
            Let $X$ be a compact, non-orientable metric $n$-manifold with finite Hausdorff $n$-measure and (possibly empty) boundary. Suppose that $X$ has finite Nagata dimension and $\Ha^n(\partial X) = 0$. Then, $X$ has a metric fundamental class $T \in \cI_n(X)$ modulo $2$ satisfying $ \norm{T}_2 = \norm{T} \leq C\Ha^n$ and whenever $S \in \bI_n(X)$ satisfies $\spt_2 (\partial S) \subset \partial X$, then $S = k \cdot T \mod 2$ for $k$ equal to either $1$ or $0$. 
        \end{theorem}

        Here, $C$ depends only on $n$. We note that a metric manifold with finite Nagata dimension does not need to be rectifiable. Indeed, by \cite[Theorem A.1]{Sormani-Wenger-calcvar-2010} there exists a geodesic metric space $Y$ with finite Hausdorff $2$-measure that is homeomorphic to the $2$-sphere but is not $2$-rectifiable. It follows from \cite{Jorgensen-Lang-22} that any geodesic metric surface, and in particular $Y$, has Nagata dimension $2$. Therefore, in general, an integer rectifiable current as in the previous result cannot satisfy a lower bound of the form \eqref{eq: mass-measure-hausdorff-existence-weak-llc-non-oriented}. Finally, for metric $2$-manifolds, called metric surfaces, we do not require any conditions. 
    
        \begin{theorem}\label{thme: existence-non-orient-surface}
            Let $X$ be a compact, non-orientable metric surface with finite Hausdorff $2$-measure and (possibly empty) boundary. Suppose that $\Ha^2(\partial X)= 0$. Then, $X$ has a metric fundamental class $T \in \cI_2(X)$ modulo $2$ satisfying $\norm{T}_2\leq 2\Ha^2$ and whenever $S \in \bI_2(X)$ satisfies $\spt_2 (\partial S) \subset \partial X$, then $S = k \cdot T \mod 2$ for $k$ equal to either $1$ or $0$. 
        \end{theorem}

        For closed, orientable surfaces, the result was already known \cite[Theorem 1.3]{BMW}.

\medskip

        We define the (lower) \textit{Minkowski content} of a subset $E\subset X$ with respect to an open subset $U\subset X$ by 
        $$
        \mathscr{M}_-(E\, |\ U)= \liminf_{r\searrow 0}\,\frac{\Ha^n(E_r\cap U) - \Ha^n(E\cap U)}{r},
        $$ 
        where $E_r=\{x\in X: d(x,E)<r\}$ denotes the open $r$-neighborhood of $E$ in $X$. For an open ball $B=B(x,r)$ and $\lambda>0$, we write $\lambda B$ for the open ball $B(x,\lambda r)$ with the same center $x$ and radius $\lambda r$. 

        \begin{theorem}\label{thme: poincare}
            Let \(X\) be a closed, non-orientable metric manifold of dimension \(n\geq 2\). Suppose that \(X\) is linearly locally contractible and Ahlfors \(n\)-regular. Then there exist \(C\), \(\lambda \geq 1\) such that 
            \begin{equation}\label{eq: rel-isop-intro}
                \min\bigl\{\Ha^n\big(E\cap B\big), \Ha^n\big(B \setminus E\big) \bigr\} \leq C \cdot\mathscr{M}_-\big( E\, | \, \lambda B\big)^{\,\tfrac{n}{n-1}}
            \end{equation}
            for every Borel subset \(E\subset X\) and every open ball \(B\subset X\).
        \end{theorem}

        Here, the constants $C$ and $\lambda$ depend only on the data of $X$. The inequality \eqref{eq: rel-isop-intro} is called a relative isoperimetric inequality and is strongly related to the Poincar\'e inequality. In particular, it follows that a metric manifold, as in Theorem \ref{thme: poincare}, also satisfies a weak $1$-Poincar\'e inequality; see \cite{Bobkov-Houdre-1997, Kinnunen-Korte-Shanmugalingam-Tuominen-2012, Korte-Lahti-2014}.

    \subsection{Structure of the article}
        In Section \ref{sec: notions}, we introduce the notation and basic results we need throughout the article. We review the theory of metric currents in Section \ref{sec: currents} and give a detailed overview of flat currents modulo $p$. There, we also prove a compactness and rectifiability result for such currents. 
        
        \medskip

        We present the new measure theoretic version of linear local contractibility, which we call almost everywhere linear local contractibility in Section \ref{sec: llc}. In addition, we explain the construction of the orientable double cover of a metric manifold and how one can glue two copies of the manifold along their boundaries. For both spaces, we construct a metric using the metric of the original manifold. The main motivation behind the new version of linear local contractibility is that it interacts well with these constructions. Indeed, if $X$ is linearly locally contractible, then the manifold $\hat{X}$ obtained by gluing two copies along their boundaries does not need to be linearly locally contractible. However, $\hat{X}$ will be almost everywhere linearly locally contractible in case the boundary of $X$ has zero Hausdorff $n$-measure. Finally, we discuss how to extend bi-Lipschitz maps from Euclidean space into metric manifolds in a controlled manner. These techniques allow us to show that the degree of Lipschitz maps behaves nicely in metric manifolds that are almost everywhere linearly locally contractible. Similar results were already obtained in \cite{BMW}.

        \medskip
        
        In Section \ref{sec: orientable-mflds}, we prove the existence of a metric fundamental class for orientable manifolds that are almost everywhere linearly locally contractible. We first consider only closed manifolds. In this case, the argument is very similar to the proof of \cite[Theorem 1.1]{BMW}. Indeed, by using linear local contractibility, we can equip the rectifiable part of the manifold with an orientation that is compatible with the topological manifold. This "metric orientation" induces an integer rectifiable current $T$ satisfying the upper measure bound in \eqref{eq: mass-measure-hausdorff-existence-weak-llc}. It follows from a deep theorem of Bate, Theorem \ref{thme: bate-perturb}, and degree theory, that $T$ is a cycle and Definition \ref{def: metric fundamental class}(a) holds. The result for manifolds with boundary can easily be reduced to the closed case with the constructions in Section \ref{sec: llc}.

        \medskip

        The proofs of the different existence results of the metric fundamental class modulo $2$ are contained in Section \ref{sec: non-orient-mflds}. We begin with the case when the manifold $X$ is closed and almost everywhere linearly locally contractible. By construction, the orientable double cover $\tilde{X}$ is orientable and locally isometric to the original manifold. As a consequence, $\tilde{X}$ is almost everywhere linearly locally contractible as well and thus, has a metric fundamental class $\tilde{T}$. We construct an integral current $T$ in $X$ from $\tilde{T}$ that satisfies the measure bound \eqref{eq: mass-measure-hausdorff-existence-weak-llc-non-oriented}. To show that $T$ is a cycle modulo $2$ and satisfies \eqref{eq: def-metric-fundamental-class-modulo 2} we use degree theory and the techniques introduced in Section \ref{sec: llc}. In case $X$ is closed and has finite Nagata dimension, we use an approximation of $X$ by a simplicial complex. This is a construction frequently used in the context of Nagata dimension; see e.g. \cite{basso2021undistorted, lang-2005, marti2024characterization}. More precisely, there exists a simplicial complex $\Sigma$ with bounded Hausdorff measure and Lipschitz maps $\varphi \colon X \to \Sigma$, and $\psi\colon \Sigma\to E(X)$. Furthermore, the composition $(\psi \circ\varphi) \colon X \to E(X)$ is arbitrarily close to the inclusion of $X$ into its injective hull $E(X)$. The space $E(X)$ has similar properties as $l^\infty$ but is compact; see Section \ref{sec: notions-1-part}. We then approximate the composition $(\varphi\circ \varrho)$ by a Lipschitz map, where $\varrho$ denotes the homeomorphism from the smooth manifold $M$ into $X$. Using this approximation, we can push the fundamental class modulo $2$ of $M$ into $\Sigma$. In this way, we construct a sequence of integral cycles modulo $2$ with uniformly bounded $2$-mass. It follows from the compactness and rectifiability theorem for flat currents modulo $2$ established in Section \ref{sec: currents} that $X$ has a metric fundamental class modulo $2$. For non-orientable closed metric surfaces $X$, we use a recent uniformization result \cite{nta-rom22}. This result implies that there exists a weakly conformal map from a smooth surface into $X$. We can approximate this map by Lipschitz maps with bounded area. As before, we can use a pushforward argument and find a metric fundamental class modulo $2$ in $X$. Finally, we show the existence of the metric fundamental class for non-orientable manifolds with boundary by reducing the proof to the corresponding case for closed manifolds.

\section{Preliminaries}\label{sec: notions}
    \subsection{Metric notions}\label{sec: notions-1-part}
        Let $(X,d)$ be a metric space. Given $x \in X$ and $r>0$, we write $B(x,r)$ for the open ball with center $x \in X$ and radius $r> 0$. The infimal distance between two subsets $A$ and $B$ is defined by 
        $$d(A, B)=\inf\big\{ d(a, b) : a\in A \text{ and } b\in B\big\}.$$
        Given $A \subset X$ and $r>0$, we denote the open $r$-neighborhood of $A$ by 
        $$N_r^X(A) = \left\{x \in X \colon d(A,x) < r\right\}.$$
        If the ambient space $X$ is clear from the context, we simply write $N_r(A)$. We call a map $f\colon X \to Y$ between metric spaces \textit{\(L\)-Lipschitz} if $d(f(x),f(y)) \leq L d(x,y)$ for all $x,y \in X$. The smallest constant that satisfies this inequality is called the \textit{Lipschitz constant} of $f$ and is denoted by $\Lip(f)$. In case $f$ is injective and its inverse $f^{-1}$ is Lipschitz as well, we say $f$ is \textit{bi-Lipschitz}. We write $\LIP(X,Y)$ for the set of all Lipschitz maps from $X$ to $Y$. If $Y= \R$ we abbreviate $\LIP(X,\R)=\LIP(X)$. The uniform distance between two maps \(f, g\colon X \to Y\) is given by
        \[
        d(f, g)=\sup\big\{ d(f(x), g(x)) : x\in X\big\}.
        \]
        The following result is easy to verify; see e.g. \cite[Lemma 2.1]{BMW}.

        \begin{lemma}\label{lemma: lip-approx-easy}
            Let \(f\colon X \to Y\) be a continuous map from a compact metric space \(X\) to a separable metric space \(Y\) that is an absolute Lipschitz neighborhood retract. Then, for every \(\epsilon>0\) there exists a Lipschitz map \(g \colon X \to Y\) with \(d(f, g)<\epsilon\). 
        \end{lemma}

        Here, $Y$ is said to be an absolute Lipschitz neighborhood retract if there exists $C>0$ such that the following holds. Whenever $Y$ is a subset of another metric space $Z$ then there exists an open neighborhood $U$ of $Y$ in $Z$ and a $C$-Lipschitz retraction $\pi \colon U \to Y$. Note that every closed Riemannian $n$-manifold is an absolute Lipschitz neighborhood retract; see e.g. \cite[Theorem 3.1]{hohti-1993}. We also need the following result about Lipschitz neighborhood retracts.

        \begin{lemma}\label{lemma: lip-retract-homotopies}
            Let $Y$ be a compact metric space that is an absolute Lipschitz neighborhood retract. Then there exist $C>0$ and $\epsilon>0$ with the following property. If $f,g\colon X \to Y$ are Lipschitz maps from a compact metric space $X$ into $Y$ that satisfy $d(f,g)< \epsilon$, then there exists a Lipschitz homotopy $H \colon [0,1]\times X \to Y$ between $f$ and $g$ satisfying $d(H(t,x),f(x)) \leq C d(f,g)$ for every $x \in X$ and each $t \in [0,1]$.
        \end{lemma}

        \begin{proof}
            We embed $Y \subset l^\infty$. Since $Y$ is an absolute Lipschitz neighborhood retract there exists a $C$-Lipschitz retraction $\pi \colon N_\epsilon^{l^\infty}(Y)\to Y$ for some $C>0$ and $\epsilon>0$. Let $h\colon [0,1] \times X \to l^\infty$ be the straight-line homotopy between $f$ and $g$. Notice that $h(t,x)\in N_\epsilon(Y)$ for every $x \in X$ and each $t \in [0,1]$. Therefore, $H= \pi \circ h$ is well-defined and satisfies the desired properties. 
        \end{proof}

        Finally, we introduce injective metric spaces and the injective hull. A metric space $X$ is said to be injective if the following holds. Whenever $f \colon A \to X$ is $1$-Lipschitz and $A$ is a subset of a metric space $B$, then there exists a $1$-Lipschitz extension $\overline{f}\colon B \to X$ of $f$. Important examples of injective metric spaces are $\R$ and $l^\infty$. Moreover, for every metric space $X$, there exists a minimal injective metric space, called the injective hull $E(X)$ of $X$, such that $X$ can be embedded isometrically into $E(X)$. Here, minimal means that any isometric embedding $X \to Y$ into an injective space $Y$ admits an isometric extension $E(X) \to Y$. If $X$ is compact, then $E(X)$ is compact as well and can be realized as a subset of $l^\infty$. We refer to \cite{lang-injective} for more information.

    \subsection{Orientation and degree}\label{sec: orientation-degree}
        We review the basic definitions and properties of topological orientation and degree. For a detailed exposition of the theory, see \cite{dold-1980}. Let $X$ be a topological $n$-manifold without boundary but not necessarily compact. Notice that throughout this article, all manifolds are assumed to be connected. For $x\in X$, we denote the $n$-th local singular homology group with coefficients in $\Z$ by $H_n(X,X\setminus x)$. It follows from the excision theorem that $H_n(X,X\setminus x)$ is infinite cyclic. A \textit{local orientation} $o_x$ at $x$ is a generator of $H_n(X,X\setminus x)$. Furthermore, an orientation of $X$ is a choice of local orientations $o_x$ for each $x\in X$ that satisfies a certain continuity condition. In case an orientation exists, $X$ is said to be orientable, and $X$ together with an orientation is called oriented. It follows that if $X$ is an oriented manifold, then for each connected, compact subset $K \subset X$ the group $H_n(X,X\setminus K)$ is infinite cyclic. Moreover, there exists a generator $o_K\in H_n(X,X\setminus K)$ such that for each $x\in K$ the homomorphism from $ H_n(X,X\setminus K)$ to $ H_n(X,X\setminus x)$ induced by the inclusion sends $o_K$ to $o_x$. For a compact oriented $n$-manifold $X$ we call $o_X$ the \textit{fundamental class} of $X$ and denote it by $[X]$. Now, let $f\colon X \to Y$ be a continuous map between two oriented $n$-manifolds without boundary. Suppose that $K \subset Y$ is a compact connected subset such that $f^{-1}(K)$ is compact as well. Then $f_*\colon H_n(X,X\setminus f^{-1}(K)) \to H_n(Y,Y\setminus K)$ is a well-defined homomorphism and sends $o_{f^{-1}(K)}$ to an integer multiple of $o_K$. This integer is denoted by $\deg_K(f)$ and called the \textit{degree} \textit{of} $f$ \textit{over} $K$. If $X$ is compact, then $\deg_K(f)$ is the same number for each compact, connected subset $K\subset X$ and we write $\deg(f)$. We call $\deg(f)$ the \textit{degree} \textit{of} $f$. It is a homotopy invariant. More precisely, if $g\colon X \to Y$ is another continuous map that is homotopic to $f$, then $\deg(f) = \deg(g)$. If $Y$ is also compact, then the degree is characterized by the following equality $f_*[X]=\deg(f) \cdot [Y]$. Next, we define the local degree. Let $x\in X$ and let $U \subset X$ be an open neighborhood of $x$. By excision, the homomorphism $H_n(U,U\setminus x) \to H_n(X,X\setminus x)$ is an isomorphism. If $V\subset Y$ is an open neighborhood of $f(x)$ such that $f(U) \subset V$ and $f(y) \neq f(x)$ for all $y \in U\setminus x$, then $f_* \colon H_n(U,U\setminus x) \to H_n(V,V\setminus f(x))$ is well-defined. This homomorphism sends the local orientation $o_x$ to an integer multiple of the local orientation $o_{f(x)}$. We call this number the \textit{local degree of $f$ at} $x$ and denote it by $\deg(f,x)$. Notice that the local degree does not depend on $U$ and $V$. The local degree satisfies the following multiplicity property; see \cite[Chapter 8, Corollary 4.6]{dold-1980}. If $h \colon Y \to Z$ is a continuous map into an oriented topological $n$-manifold $Z$ without boundary and the local degrees $\deg(f,x)$ and $\deg(h,f(x))$ are well-defined, then
        $$\deg(h\circ f, x) = \deg(h,f(x)) \cdot \deg(f,x).$$
        We have the following relation between the degree and the local degree. Suppose that $y \in Y$ is a point such that $f^{-1}(y)$ is finite. Then the local degree of $f$ at each $x \in f^{-1}(y)$ is well-defined and we have
        $$\deg(f) = \sum_{x \in f^{-1}(y)} \deg(f,x).$$
        We refer to this property as the additivity property of the degree; see \cite[Chapter 8, Proposition 4.7]{dold-1980}. A topological $n$-manifold with boundary $X$ is said to be orientable if its interior $X \setminus \partial X$ is orientable. In case $X$ is compact and oriented, it follows that $H_n(X, \partial X)$ is infinite cyclic and the orientation of $X$ induces a generator $[X]$ of $H_n(X, \partial X)$. Again, we call $[X]$ the fundamental class of $X$. Let $f\colon X \to Y$ be a continuous map between two compact, oriented $n$-manifolds with boundary such that $f(\partial X) \subset \partial Y$. Then, the \textit{degree $\deg(f)$ of $f$} is the unique integer such that the homomorphism $f_* \colon H_n(X,\partial X)\to H_n(Y,\partial Y)$ satisfies $f_*[X]=\deg(f) \cdot [Y]$. Finally, we note that it is possible to define orientability and the degree using singular homology groups with coefficients in $\Z/2\Z$ instead of $\Z$ in the same manner as before; see \cite[Chapter 8, Definition 4.1]{dold-1980}. Every manifold is orientable in this sense. Furthermore, the degree obtained in this way is an element of $\Z/2\Z$ and satisfies all properties analogous to those explained above.
        
    \subsection{Rectifiability}
        Let $X$ be a complete metric space. We denote by $\Ha^n$ the $n$-dimensional Hausdorff measure on $X$, which is normalized such that $\Ha^n$ is equal to the Lebesgue measure $\leb^n$ on $\R^n$. A $\Ha^n$-measurable subset $E \subset X$ is called \textit{$n$-rectifiable} if there exist compact sets $K_i \subset \R^n$ and bi-Lipschitz maps $\varphi_i \colon K_i \to X$ such that 
        $$\Ha^n\left( E \setminus \bigcup_i \varphi_i(K_i)\right) =0.$$
        Notice that this definition is not the standard definition, but using \cite[Lemma 4.1]{ambrosio-kirchheim-2000} it can easily be shown that the definitions are equivalent. We say a $\Ha^n$-measurable subset $P \subset X$ is \textit{purely $n$-unrectifiable} if $\Ha^n(P \cap E)=0$ for every $n$-rectifiable subset $E \subset X$. We need the following deep result due to Bate.

        \begin{theorem}\label{thme: bate-perturb}{\normalfont (\cite[Theorem~1.2]{bate24})}
            Let $X$ be a complete metric space and let $P \subset X$ be purely $n$-unrectifiable with finite Hausdorff $n$-measure. Then, for every $m \in \N$ the set 
            $$\{f \in \LIP_1(X,\R^m) \colon \Ha^n(f(P)) = 0\}$$
            is residual. 
        \end{theorem}

        Here, $\Lip_1(X,\R^m)$ denotes the space of all Lipschitz maps $X\to \R^m$ with Lipschitz constant at most $1$. The theorem was first proven in \cite{bate-perturb} with the additional assumption that the lower density of $P$ is positive at almost every point. However, \cite[Theorem 1.5]{bate24} shows that this assumption is not necessary. The next result can be seen as a strong converse of the previous result in Euclidean space.

        \begin{theorem}\label{thme: bate-jakub}{\normalfont (\cite[Theorem~1.1]{bate-jakub})}
            Let $E\subset\R^k$ be $n$-rectifiable and $m>n$. Then the set
            $$\{f \in \LIP_1(\R^k,\R^m) \colon \Ha^n(f(E)) = \Ha^n(E)\}$$
            is residual. 
        \end{theorem}

         If $X$ is a complete metric space, then $\LIP_1(X,\R^m)$ equipped with the uniform distance is a complete metric space as well. In this case, the Baire category theorem implies that every residual subset of $\LIP_1(X,\R^m)$ is dense.
        
    \subsection{Nagata dimension}
        The Nagata dimension can be seen as a quantitative version of the topological dimension that is better adapted to Lipschitz analysis. It has recently gained attention in the study of metric spaces, most prominently in the context of Lipschitz extension problems; see e.g. \cite{basso2021undistorted, GCDavid-Nagata, lang-2005, lang2025isoperimetric}.
    \medskip
    
        We say a covering of a metric space $X$ has \textit{\(s\)-multiplicity} at most \(N\) if each subset of \(X\) with diameter less than \(s\) intersects at most \(N\) members of the covering. 

        \begin{definition}\label{def: nagata-dimension}
            A metric space has Nagata dimension $\dim_N X\leq N$ if there exists $c>0$ such that for every $s>0$ there exists a covering $\{U_i\}_{i\in I}$ of $X$ with $s$-multiplicity at most $N+1$ and $\diam U_i<cs$ for every $i \in I$.
        \end{definition}

        The Nagata dimension is always at least the topological dimension and at most the Assouad dimension of a space \cite{Enrico-nagata}. In particular, every doubling metric space has finite Nagata dimension. 
        
        \medskip

        Let $\Sigma$ be a finite simplicial complex and denote by $I$ the set of vertices of $\Sigma$. Then, up to rescaling, $\Sigma$ can be realized as a subcomplex of
        $$\Sigma(I) = \left\{ x  \in l^2(I) \colon x_i\geq 0 \textup{ and } \sum_{i\in I} x_i =1 \right\}.$$
        The \textit{$l^2$-metric} on $\Sigma$ is the metric induced by the $l^2$-norm on $l^2(I)$ and is denoted by $|\cdot|_{l^2}$. For more details on simplicial complexes, we refer to \cite{basso2021undistorted}. We need the following factorization theorem. The result is essentially \cite[Proposition 6.1]{basso2021undistorted} combined with a Federer-Fleming deformation type theorem.

        \begin{theorem}\label{thme: nagata-factorization}
                Let $X$ be a compact metric space with finite Hausdorff $n$-measure and finite Nagata dimension. Then there exists a constant $C>0$, depending only on the data of $X$, such that for every $\epsilon>0$ there exist a finite simplicial complex $\Sigma$ equipped with the $l^2$-metric and Lipschitz maps $\varphi\colon X \to \Sigma$, $\psi\colon \Sigma \to E(X)$ with the following properties
             \begin{enumerate}
                 \item $\psi$ is $C$-Lipschitz on every simplex in $\Sigma$;
                 \item $\Ha^n(\varphi(X)) \leq C \Ha^n(X);$
                 \item $\Sigma$ has dimension $\leq n$ and every simplex in $\Sigma$ is a Euclidean simplex of side length $\epsilon$;
                 \item $\textup{Hull}(\varphi(X)) = \Sigma$ and $d(x,\psi(\varphi(x))) \leq C \epsilon$ for all $x \in X$.
            \end{enumerate}
        \end{theorem}

        Here, $\textup{Hull}(A)$ denotes the \textit{hull} of $A$ in $\Sigma$ and is the smallest subcomplex of $\Sigma$ containing $A$. 
        
        \begin{proof}
            Let $\epsilon>0$. Since $E(X)$ is an injective metric space, it follows directly from \cite[Proposition 6.1]{basso2021undistorted} and a simple rescaling argument that there exist a finite simplicial complex $\Sigma'$ of dimension $\leq \dim_N X$ and every simplex in $\Sigma'$ is a Euclidean simplex of side length $\epsilon$ and there are Lipschitz maps $\varphi'\colon X \to \Sigma'$, $\psi\colon \Sigma' \to E(X)$ such that $\varphi'$ is $C$-Lipschitz, $\psi$ is $C$-Lipschitz on each simplex of $\Sigma'$ and $d(x,\psi(\varphi'(x)))\leq C \epsilon$ for each $x \in X$. Here, $\Sigma'$ is equipped with the $l^2$-metric $|\cdot|_{l^2}$ and $C$ depends only on the data of $X$. Furthermore, \cite[Proposition 5.3]{marti2024characterization} implies that there exists a map $p\colon \Sigma' \to \Sigma'$ such that composition $\varphi= p\circ \varphi'$ is Lipschitz and has the following properties; the image $\varphi(X)$ is contained in the $n$-skeleton of $\Sigma'$, $\varphi$ satisfies (2) and $\varphi(x),\varphi'(x)$ are contained in a common simplex for all $x\in X$. Notice that the proposition actually proves a stronger version of (2). Now, let $x \in X$. Then, there exist two $n$-dimensional simplices $\Delta_1, \Delta_2 \subset \Sigma'$ with $\Delta_1\cap \Delta_2 \neq \emptyset$ and such that $\varphi'(x) \in \Delta_1$ and $\varphi(x) \in \Delta_2$. It follows from Lemma \cite[Lemma 2.3]{marti2024characterization} that there exists $z \in \Delta_1\cap \Delta_2$ satisfying
            $$|\varphi'(x)-z|_{l^2}+|z-\varphi(x)|_{l^2} \leq 4 \sqrt{n} |\varphi'(x)-\varphi(x)|_{l^2}.$$
            Since each simplex in $\Sigma'$ has side-length $\epsilon$ and $\psi$ is $C$-Lipschitz on each simplex, we conclude that
            $$d(\psi(\varphi'(x)),\psi(\varphi(x))) \leq d(\psi(\varphi'(x)),\psi(z)) + d(\psi(z),\psi(\varphi(x))) \leq 8 \sqrt{n} C\epsilon.$$
            Therefore, 
            $$d(x,\psi(\varphi(x))) \leq d(x,\psi(\varphi'(x)))+ d(\psi(\varphi'(x)),\psi(\varphi(x))) \leq (1 +8 \sqrt{n})C \epsilon,$$
            for all $x \in X$. Finally, we pick a point $z_i \in \inte\Delta_i \setminus \varphi(X)$ for each $n$-simplex $\Delta_i \subset \Sigma'$ whose interior is not fully covered by $\varphi$. We compose $\varphi$ with a radial projection on each $\Delta_i$ with $z_i$ as the projection center (we denote the resulting map still by $\varphi$). The composition is again Lipschitz and we do not increase $\Ha^n(\varphi(X))$ because the projections centers are not in the image of $\varphi$. Clearly, $\textup{Hull}(\varphi(X))$ is equal to the $n$-skeleton of $\Sigma'$. The same argument as above, using \cite[Lemma 2.3]{marti2024characterization}, shows that $\varphi$ still satisfies the second part of (4). This completes the proof with $\Sigma$ equal to the $n$-skeleton of $\Sigma'$.
        \end{proof}

\section{Metric currents}\label{sec: currents}
    \subsection{Basic definitions}
        In this section we present the theory of metric currents. For more details we refer to \cite{ambrosio-kirchheim-2000} and \cite{lang-currents-2011}. Throughout this section let $X$ be a complete metric space. For $k\in \N$, we let $\mathcal{D}^k(X)= \LIP_b(X) \times \LIP(X)^k$, where $\LIP_b(X)$ denotes the set of bounded Lipschitz functions $X \to \R$.

        \begin{definition}\label{def:metric-current}
            A multilinear map \(T\colon \mathcal{D}^k(X)\to \R\) is called metric \(k\)-current (of finite mass) if the following holds:
            \begin{enumerate}
                \item(continuity) If \(\pi_i^{j}\) converges pointwise to \(\pi_i\) for every \(i=1, \ldots, k\), and \(\Lip \pi_i^{j}< C\) for some uniform constant \(C>0\), then 
                \[
                T(f, \pi_1^{j}, \ldots, \pi_k^{j})\to T(f, \pi_1, \ldots, \pi_k)
                \] as \(j\to \infty\);
                \item(locality) \(T(f, \pi_1, \ldots, \pi_k)=0\) if for some \(i\in\{1, \ldots, k\}\) the function \(\pi_i\) is constant on \(\{ x\in X : f(x) \neq 0\}\);
                \item(finite mass) There exists a finite Borel measure \(\mu\) on \(X\) such that
                \begin{equation}\label{eq:mass-inequality}
                |T(f, \pi_1, \ldots, \pi_k)| \leq \prod_{i=1}^k \Lip(\pi_i) \int_{X} |f(x)| \, d\mu(x)
                \end{equation}
                for all \((f, \pi_1, \ldots, \pi_k)\in \mathcal{D}^k(X)\).
            \end{enumerate}
        \end{definition}

        The minimal measure \(\mu\) satisfying \eqref{eq:mass-inequality} is called the \textit{mass measure} of \(T\) and is denoted by \(\norm{T}\). We define the \textit{support} of \(T\) as
        \[
        \spt T=\big\{ x\in X : \norm{T}(B(x, r)) >0 \text{ for all } r>0\big\}.
        \] 
        We write $\bM_k(X)$ for the vector space of all metric $k$-currents on $X$. Endowed with the mass norm $\bM(T) =\norm{T}(X)$, it is a Banach space. Given a Borel set $B \subset X$, the \textit{restriction} of $T$ to $B$ is defined by 
        $$(T\on B)(f, \pi_1\ldots, \pi_k)=T(\mathbbm{1}_B \cdot f, \pi_1, \ldots, \pi_k).$$
        This is a well-defined metric \(k\)-current because each metric $k$-current can be extended uniquely to $L^1(\norm{T}) \times \LIP^k(X)$. The restriction satisfies \(\norm{T\on B}=\norm{T}\on B\). A sequence of metric $k$-currents $T_i\in \bM_k(X)$ is said to converge weakly to $T \in \bM_k(X)$ if
        $$\lim_{i \to \infty}T_i(f,\pi_1,\dots,\pi_k) = T(f,\pi_1,\dots,\pi_k)$$
        for all $(f,\pi_1,\dots, \pi_{k})\in \mathcal{D}^k(X)$ and we write $T_i \rightharpoonup T$. For $T \in \bM_k(X)$, the \textit{boundary} of $T$ is the multilinear map on $\mathcal{D}^{k-1}(X)$ given by
        $$\partial T (f,\pi_1,\dots,\pi_{k-1}) = T(1,f,\pi_1,\dots,\pi_{k-1})$$
        for all $(f,\pi_1,\dots,\pi_{k-1})\in \mathcal{D}^{k-1}(X)$. If $\partial T$ satisfies \eqref{eq:mass-inequality}, then it is a metric $(k-1)$-current, and we call $T$ a normal $k$-current. The set of all normal $k$-currents in $X$ is denoted by $\bN_k(X)$ and we define the normal mass by $\bN(T) = \bM(T) + \bM(\partial T)$. Again, $\bN_k(X)$ equipped with the normal mass is a Banach space.
        The \textit{pushforward} of $T\in \bM_k(X)$ under a Lipschitz map $\varphi\colon X \to Y$ is the metric $k$-current in $Y$ given by
        $$\varphi_\# T(f,\pi_1,\dots,\pi_k) = T(f\circ \varphi,\pi_1\circ \varphi,\dots,\pi_k\circ\varphi)$$
        for all $(f,\pi_1,\dots, \pi_{k})\in \mathcal{D}^k(Y)$. It is not difficult to verify that the pushforward satisfies the following properties: $(\varphi_\#T)\on B = \varphi_\#(T\on\varphi^{-1}(B))$, the support of $\varphi_\#T$ is contained in the closure of $\varphi(\spt T)$ and $\bM(\varphi_\# T) \leq \Lip(\varphi)^k \bM(T)$.
        
    \subsection{Integer rectifiable and integral currents}
        In Euclidean space, every metric current is given by integration with respect to a $L^1$ function as follows. For $\theta \in L^1(\R^k)$, the metric $k$-current $\bb{\theta} \in \bM_k(\R^k)$ is defined by
        $$\bb{\theta}(f,\pi_1,\dots,\pi_k) = \int_{\R^k} \theta f \det(D \pi) \; d\leb^k$$
        for all $(f,\pi) = (f,\pi_1,\dots,\pi_k) \in \mathcal{D}^k(\R^k)$. Notice that this characterization of metric $k$-currents in $\R^k$ is a very deep result; see \cite{rindler-flat-conj}. An integer rectifiable current is defined in analogy to rectifiable sets, using the definition of currents in Euclidean space described above. More precisely, a $k$-current $T \in \bM_k(X)$ is said to be integer rectifiable if there exist countably many compact sets $K_i \subset \R^k$ and functions $\theta_i \in L^1(\R^k,\Z)$ with $\spt \theta_i \subset K_i$ and bi-Lipschitz maps $\varphi_i \colon K_i \to X$ such that
        $$T=\sum_{i\in \N} \varphi_{i\#}\bb{\theta_i} \quad \text{ and } \quad \mass(T)=\sum_{i\in \N} \mass(\varphi_{i\#}\bb{\theta_i}).$$
        We call a triple $\{\varphi_i,K_i,\theta_i\}$ as above a \textit{parametrization} of $T$.
        The space of integer rectifiable $k$-currents on $X$ is denoted by $\cI_k(X)$. For $T \in \cI_k(X)$, \textit{the characteristic set of} $T$ is defined as
        \begin{equation}\label{eq: deg-charact-set}
            \textup{set}T= \left\{ x \in X \colon \liminf_{r\to 0}\frac{\norm{T}(B(x,r))}{r^k}>0\right\}.
        \end{equation}
       By \cite[Theorem 4.6]{ambrosio-kirchheim-2000}, the characteristic set of $T$ is $k$-rectifiable and $\norm{T}$ is concentrated on $\textup{set}T$. Now, suppose that there exists $C>0$ such that $T$ has a parametrization $\{\varphi_i,K_i,\theta_i\}$ satisfying $|\theta_i(x)| \leq C$ for almost all $x \in K_i$ and every $i \in \N$. It follows from \cite[Lemma 9.2 and Theorem 9.5]{ambrosio-kirchheim-2000} that 
        $$C^{-1} \lambda \Ha^k\on\textup{set}T \leq \norm{T} \leq C \lambda \Ha^k\on\textup{set}T,$$
        where $\lambda\in L^1(X)$ satisfies $k^{-k/2}\leq \lambda \leq k^{k/2}$. In particular, we have $\bM(T) \leq C k^{k/2} \Ha^k(\textup{set} T)$. We say an integer rectifiable current is an integral current if it is also a normal current and let $\bI_k(X) = \cI_k(X) \cap \bN_k(X)$ be the space of integral $k$-currents on $X$. If an integral current $T$ has zero boundary $\partial T = 0$, then we call $T$ an integral cycle. The boundary-rectifiability theorem of Ambrosio and Kirchheim \cite[Theorem 8.6]{ambrosio-kirchheim-2000} implies that if $T\in \bI_k(X)$, then $\partial T \in \bI_{k-1}(X)$. Therefore,
        \[
        \dotsm \overset{\partial_{k+1}}{\longrightarrow} \bI_k(X) \overset{\partial_{k}}{\longrightarrow} \bI_{k-1}(X) \overset{\partial_{k-1}}{\longrightarrow} \dotsm \overset{\partial_1}{\longrightarrow} \bI_0(X) 
        \]
        is a chain complex. For $k \geq 0$, we define the $k$th homology group of this chain complex by $H_k^\textup{IC}(X) =\ker \partial_k / \im \partial_{k+1}$. Finally, let $M$ be a closed, oriented Riemannian $n$-manifold and let 
        \[
        \bb{M}(f,\pi) = \int_Mf\det(D\pi)\,d\hspace{-0.14em}\Ha^n
        \]
        for all \((f, \pi)\in \mathcal{D}^n(M)\). This defines an integral $n$-current on $M$ and by Stokes' theorem $T$ is a cycle. Moreover, this cycle satisfies $\norm{\bb M}=\Ha^n$ and generates the $n$-th homology group via integral currents $H_n^\textup{IC}(X)$. We thus refer to $\bb{M}$ as the fundamental class of $M$.
        

     \subsection{Isoperimetric inequality}
        Let $T \in \bI_k(X)$ be an integral $k$-cycle. We say $U \in \bI_{k+1}(X)$ is a filling of $T$ if $\partial U = T$. Moreover, if
        $$\bM(U) =\inf \{\bM(V) \colon V \in \bI_{k+1}(X), \partial V=T\} $$
       then $U$ is called a minimal filling of $T$. The isoperimetric inequality guarantees the existence of fillings that satisfy a certain mass bound. More precisely, we say $X$ satisfies a $k$-dimensional (Euclidean) isoperimetric inequality if there exists $D_k>0$ such that for every $T \in \bI_{k}(X)$ with $\partial T = 0$, there exists a filling $S \in \bI_{k+1}(X)$ of $T$ satisfying $\bM(S) \leq D_k \bM(T)^{ \frac{k+1}{k} }$. In \cite{wenger-isop}, Wenger proved that metric spaces that satisfy a certain cone inequality also satisfy the isoperimetric inequality for integral currents. This includes, for example, all Banach spaces and injective metric spaces. We refer to \cite[Theorem 3.4]{BMW} for the version of the isoperimetric inequality presented here. 

        \begin{theorem}\label{thme: filling-inj-spaces}
            Let \(k\geq 0\) be an integer. Then there exists a constant \(D=D_k\geq 1\) such that the following holds. If \(Y\) is an injective metric space and \(T\in \bI_{k}(Y)\) a cycle, then there exists a minimal filling \(U\in \bI_{k+1}(Y)\) of \(T\) and any such filling satisfies
            \begin{equation}\label{eq:filling-estis}
            \mass(U) \leq D \mass(T)^{\frac{k+1}{k}} \quad \text{ and } \quad \norm{U}(B(y, r))\geq D^{-k}r^{k+1}
            \end{equation}
            for each \(y\in \spt U\) and all \(r \in (0, d(y, \spt T))\). In particular, \(\spt U\) is contained in the \(R\)-neighborhood of \(\spt T\) for \(R=D\mass(U)^{\frac{1}{k+1}}\).
        \end{theorem}
        
        If \(k=0\) then the first inequality in \eqref{eq:filling-estis} is understood as an empty statement.

    \subsection{Flat currents modulo $p$}\label{sec: flat-currents}
        We follow the theory of flat currents modulo $p$ introduced in \cite{Ambrosio-Katz} and \cite{Ambrosio-Wenger}. For $k\geq 0$, we denote by $\cF_k(X)$ the space of flat currents in $X$, that is, the currents that can be written as $T = R + \partial S$, where $R \in \cI_k(X)$ and $S\in \cI_{k+1}(X)$. Clearly, this defines an additive group. If $T$ is a flat $k$-current of the form $T = R + \partial S$ with  $R \in \cI_k(X)$ and $S\in \cI_{k+1}(X)$, then its boundary $\partial T$ is the flat $(k-1)$-current given by $\partial T = \partial R$. Notice that a flat current does not need to have finite mass as defined in \eqref{eq:mass-inequality}. Thus, the mass norm does not provide a good notion of distance in this space. Instead, we consider the flat norm
        $$\cF(T) = \inf\{ \bM(R) +\bM(S) \colon T = R +\partial S, R \in \cI_k(X), S\in \cI_{k+1}(X)\}.$$
        It follows from the subadditivity of $\cF$ and the completeness of $\cI_k(X)$ with respect to the mass norm that $\cF$ defines a complete metric on $\cF_k(X)$. We have $\cF(\partial T) \leq \cF(T)$ for all  $T\in \cF_k(X)$. Let $p>1$  be an integer. We define the flat norm modulo $p$ of $T \in \cF_k(X)$ as
        $$\cF_p(T) = \inf\{\cF(T-pQ) \colon Q \in \cF_k(X)\}.$$
        For $T,T'\in \cF_k(X)$, we write $T = T' \mod p$ whenever $\cF_p(T-T') = 0$. In particular, if there exists $S\in \cF_k(X)$ such that $T = T' + pS$, then $T=T'\mod p$. Given $T \in \cF_k(X)$, we define the \textit{(relaxed) $p$-mass of} $T$ by 
        $$\bM_p(T) = \inf\left\{  \liminf_{i \to \infty} \bM(T_i)\colon T_i \in \cI_k(X), \cF_p(T-T_i) \to 0\right\}.$$
        Clearly, $\bM_p(T) = \bM_p(T')$ if $T = T' \mod p$. A direct computation shows that the $p$-mass is subadditive and lower semi-continuous with respect to $\cF_p$-convergence. Moreover, we have $\cF_p(T) \leq \bM_p(T) \leq \bM(T)$ for all $T \in \cI_k(X)$. Let $\varphi\colon X \to Y$ be a Lipschitz map between two complete metric spaces and $T \in \cF_k(X)$. Then the pushforward $\varphi_\#T$ is a flat current in $Y$ satisfying
        $$\cF_p(\varphi_\#T) \leq \Lip(\varphi)^k  \cF_p(T) \quad \textup{and} \quad  \bM_p(\varphi_\#T) \leq \Lip(\varphi)^k \bM_p(T).$$

    \medskip

        
        Now, let $E$ be a compact and convex subset of a Banach space. Furthermore, let $T \in \cF_k(E)$ be a flat $k$-current with finite $\bM_p$-mass and let $\pi \colon E \to \R$ be Lipschitz. It was shown in \cite{Ambrosio-Katz} that for almost all $r \in \R$, the restriction $T\on \{\pi < r\}$ is a well-defined flat $k$-current and there exists a finite, non-negative and $\sigma$-additive Borel measure $\norm{T}_p$ such that 
        $$\bM_p(T\on \{\pi < r\}) = \norm{T}_p(\{\pi < r\})$$
        for almost all $r \in \R$. We refer to $\norm{T}_p$ as the $p$-mass measure of $T$. Exactly as in \cite[Section 2.4]{Ambrosio-Wenger}, we can extend this equality to closed and open subsets of $E$. Notice that in \cite{Ambrosio-Wenger}, it is assumed that the ambient Banach space satisfies a strong finite-dimensional approximation property. However, in the argument therein, it is only used that $E$ is compact and that $\bI_k(E)$ is dense in $\cI_k(E)$ and $\cF_k(E)$, which is true for closed and convex subsets of a Banach space; see \cite[Proposition 14.7]{Ambrosio-Katz}. For an open or closed subset $A\subset E$, the restriction $T\on A$ is a flat $k$-current satisfying
        $$\bM_p(T\on A) = \norm{T}_p(A) \quad \textup{and} \quad T = T\on A + T \on(E\setminus A).$$
        It follows from the construction that there exists a sequence $T_n \in \cI_k(E)$ such that for all closed or open sets $A \subset E$ we have $\cF_p(T_n\on A -T\on A ) \to 0$ and $\bM(T_n\on A) \to \bM(T\on A)$ as $n \to \infty$. Using this, one easily proves that $\norm{T\on A}_p = \norm{T}_p \on A$ for all closed and open sets $A \subset E$ and
        $$\norm{\varphi_\#T}_p \leq \Lip(\varphi)^k \varphi_\#\norm{T}_p$$
        for all Lipschitz maps $\varphi\colon E \to F$, where $F$ is a compact and convex subset of a Banach space. We define 
        $$\spt_p T = \{x \in E \colon \norm{T}_p(B(x,r)) >0 \textup{ for all }r>0\}.$$
        By the above, we have $\spt_p(\varphi_\#T) \subset \varphi(\spt_p(T))$.

        \medskip
        

         We also need the reduction of an integer rectifiable current modulo $p$. Let $T \in \cI_k(X)$ and let $\{\varphi_i,K_i,\theta_i\}$ be a parameterization of $T$. A \textit{reduction modulo} $p$ \textit{of} $T$ is any integer rectifiable $k$-current $T^p$ with a parametrization $\{\varphi_i,K_i,\eta_i\}$ that satisfies $\eta_i(x) = \theta_i(x) \mod p$ for almost all $x \in K_i$ and every $i \in \N$. Notice that a reduction $T^p$ modulo $p$ of $T$ is not unique. However, we always have $T = T^p \mod p$ and $\bM(T^p)$ is uniquely determined. Since $T^p$ is given by a parametrization $\{\varphi_i,K_i,\eta_i\}$ satisfying $|\eta_i(x)| \leq p$ for almost all $x \in K_i$ and every $i \in \N$, we have $\bM(T^p) \leq p k^{k/2} \Ha^k\on\textup{set}T$. Notice that $\textup{set}T^p \subset \textup{set}T$. If $X$ is a compact length space, then \cite[Theorem 10.5]{Ambrosio-Katz} implies that $\bM_p(T) = \bM(T^p)$ for each reduction $T^p$ of $T$ modulo $p$. Using the injective hull $E(X)$, we can extend this to compact metric spaces.

         \begin{lemma}\label{lemma: mass-equal-mod-mass-for-reduction}
             Let $X$ be a compact metric space and $T \in \cI_k(X)$. Then $\bM_p(T) = \bM(T^p)$, where $T^p \in \cI_k(X)$ is any of $T$ modulo $p$. In particular, there exists a finite, non-negative and $\sigma$-additive Borel measure $\norm{T}_p$ such that
             $$\bM_p(T\on B) = \norm{T}_p(B) = \norm{T^p}(B)$$
             for every Borel set $B \subset X$.
         \end{lemma}

        The measure $\norm{T}_p$ coincides with the $p$-mass measure introduced above.
        
         \begin{proof}
            Let $\iota\colon X \to E(X)$ be an isometric embedding. Recall that $E(X)$ is compact because $X$ is compact. Furthermore, $E(X)$ is geodesic since it is an injective metric space. Let $T^p \in \cI_k(X)$ be any reduction of $T$ modulo $p$. Notice that $\iota_\# T^p$ is also a reduction modulo $p$ of $\iota_\#T$. Therefore,
            $$\norm{T^p}(E) = \norm{\iota_\#T^p}(E(X)) = \bM_p(\iota_\# T)\leq \bM_p(T).$$
            Since $T = T^p \mod p$ we have $\bM_p(T) = \bM_p(T^p) \leq \bM(T^p)$ and thus, $\bM_p(T) = \bM(T^p)$. Now, let $B \subset X$ be Borel. Then $T^p \on B$ is a reduction modulo $p$ of $(T\on B)$. Hence, the first part of the proof implies that $\bM(T^p\on B) = \norm{T^p}(B)$. It follows that $\norm{T}_p  \coloneqq \norm{T^p}$ is a well-defined, finite, non-negative and $\sigma$-additive Borel measure satisfying $\bM_p(T\on B) = \norm{T}_p(B) = \norm{T^p}(B)$ for every Borel set $B \subset X$.
         \end{proof}

         As a consequence of the previous lemma, we obtain the following equivalent characterization of the flat norm modulo $p$ in a compact metric space $X$. If $T \in \cF_k(X)$, then $\cF_p(T)$ is equal to
         $$\inf\{ \bM_p(R) +\bM_p(S) \colon T = R +\partial S + pQ, R \in \cI_k(X), S\in \cI_{k+1}(X), Q \in \cF_k(X)\}.$$  
         We conclude the section with the description of the fundamental class modulo $2$ of a closed, non-orientable Riemannian $n$-manifold $M$. Similarly as in the orientable case $M$ supports an integral current $\bb{M}\in \bI_n(M)$ satisfying $\partial \bb{M}=0 \mod 2$ and $\norm{\bb{M}}_2 = \Ha^n$; see e.g. \cite[Theorem 13.1]{Ambrosio-Katz}. We call $\bb{M}$ the fundamental class modulo $2$ of $M$. Let $\tau$ is any Borel choice of orthonormal bases spanning the tangent spaces of $M$.  Then, the integer rectifiable current defined by
            
        $$(f,\pi_1,\dots,\pi_n) \mapsto \int_M f \; \langle \tau, D\pi\rangle \; d\Ha^n$$
            
        for all $(f,\pi_1,\dots,\pi_n) \in \mathcal{D}^n(M)$ is equivalent to $\bb{M}$ modulo $2$. The fundamental class modulo $2$ is unique in the following sense. Whenever $T \in \cI_n(X)$ is an integer rectifiable current satisfying $\partial T  = 0 \mod 2$, then $T = k \cdot \bb{M} \mod 2$ for $k$ equal to either $0$ or $1$.

    \subsection{Slicing}\label{sec: slicing}
    Let $\pi \colon X \to \R$ be a Lipschitz function. For an integral current $T\in \bI_{k}(X)$, the slice of $T$ at $t \in \R$ is defined as
        \begin{equation}\label{eq: def-slice}
            \langle T,\pi ,t\rangle = \partial ( T\on\{\pi <t\}) - (\partial  T)\on\{\pi <t\}.
        \end{equation}
        It follows that for almost all $t \in \R$ the slice $\langle T,\pi ,t\rangle$ is an integral $(k-1)$-current satisfying $\spt \langle T,\pi ,t\rangle \subset \spt T \cap \pi^{-1}(t)$ and
        $$\int_r^s \norm{\langle T,\pi ,t\rangle}(B) \;dt\leq \Lip(\pi) \cdot \norm{T}(B\cap \{r<\pi<s\}),$$
        for every Borel $B \subset X$ and all $r,s\in \R$ with $r<s$; see \cite[Theorem 5.6 and Theorem 5.7]{ambrosio-kirchheim-2000}. We refer to this inequality as the \textit{slicing inequality}. We have $\partial\langle T,\pi ,t\rangle = - \langle \partial T,\pi ,t\rangle$ and in particular, if $T $ has zero boundary, then $\partial \langle T,\pi ,t\rangle =0$. We will need the slicing inequality for the $p$-mass measure.

        \begin{lemma}\label{lemma: slicing-mod-p}
            Let $E$ be a compact and convex subset of $l^\infty$. Furthermore, let $T \in \bI_k(E)$ and let $\pi\colon E \to \R$ be Lipschitz. Then, 
            $$\int_r^s \norm{\langle T,\pi ,t\rangle}_p(B) \;dt\leq \Lip(\pi) \cdot \norm{T}_p(B\cap \{r<\pi<s\}),$$
            for every Borel $B \subset X$ and all $r,s\in \R$ with $r<s$.
        \end{lemma}
        
        The lemma is known to be true in case $l^\infty$ is replaced by a separable Banach space; see \cite[Theorem 10.6]{Ambrosio-Katz}. Notice that $E(X)$ satisfies the conditions in the lemma in case $X$ is compact. 

        \begin{proof}

           Let $T^p\in \cI_k(E)$ be a reduction of $T$ modulo $p$. It is a direct consequence of \cite[Theorem 9.7]{ambrosio-kirchheim-2000} that for almost all $t \in \R$, the slice $\langle T^p,\pi ,t\rangle$ is a reduction modulo $p$ of $\langle T,\pi ,t\rangle$. Therefore, by Lemma \ref{lemma: mass-equal-mod-mass-for-reduction}, we have
           \begin{align*}
               \int_r^s \norm{\langle T,\pi ,t\rangle}_p(B)\;dt &= \int_r^s \norm{\langle T^p,\pi ,t\rangle}(B)\;dt
               \\
               &\leq \norm{T^p}(B\cap \{r<\pi<s\})
               \\
               &= \norm{T}_p(B\cap \{r<\pi<s\})
           \end{align*}
            for every Borel $B \subset X$ and all $r,s\in \R$ with $r<s$. This completes the proof.
        \end{proof}

    \subsection{Product currents and cone constructions}
        We equip $[0,1]\times E$ with the Euclidean product metric. Given a function $f\colon [0,1]\times X\to \R$ and $t\in[0,1]$, we let $f_t\colon X\to \R, \; x \mapsto f(t,x)$. For $T\in\bI_k(X)$, we define the integral current $\bb{t}\times T\in \bI_k([0,1]\times E)$ by
        $$\bb{t}\times T(f,\pi_1,\dots, \pi_k):= T(f_t, \pi_{1t},\dots,\pi_{kt})$$
        for all $(f,\pi_1,\dots, \pi_k) \in \mathcal{D}^k([0,1]\times E)$. Furthermore, the multi-linear functional $\bb{0,1}\times T$ on $\mathcal{D}^{k+1}([0,1]\times X)$ assigning 
        \begin{align*}
            (f,\pi_1,\dots, &\pi_{k+1})\mapsto 
            \\
            &\sum_{i=1}^{k+1}\int_0^1(-1)^{k+1}T\left(f_t\frac{\partial \pi_{it}}{\partial t},\pi_{1t},\dots, \pi_{(i-1)t}, \pi_{(i+1)t},\dots, \pi_{(k+1)t}\right)\,dt
        \end{align*}
        defines an element of $\bI_{k+1}([0,1]\times X)$ and satisfies 
        \begin{equation}\label{eq: homotopy-boundary}
            \partial(\bb{0,1}\times T) + \bb{0,1}\times\partial T = \bb{1}\times T - \bb{0}\times T.
        \end{equation}
        See \cite[Theorem 2.9]{ambrosio-kirchheim-2000} and \cite[Theorem 2.9]{wenger-isop}.

       \medskip 
        
        Now, assume that $E$ is a closed and convex subset of a Banach space. Then, $\bI_k(E)$ is dense in $\cI_k(E)$ and $\cF_k(E)$ with respect to the mass norm and flat norm, respectively. It follows that
        $$\cF(T) = \inf \{ \bM(R) +\bM(S) \colon T = R + \partial S, R \in \bI_k(E), S \in \bI_{k+1}(E)\} $$
        for all $T \in \cF_k(E)$. Moreover, by \eqref{eq: homotopy-boundary} we have
        $$\bb{0,1}\times (R\times \partial S) = \bb{0,1}\times R + \bb{1}\times S-\bb{0}\times S- \partial(\bb{0,1}\times S),$$
        for all $R \in \bI_k(E)$ and every $S \in \bI_{k+1}(E)$. Therefore, using a density argument, we can define $\bb{t}\times T$ and $\bb{0,1}\times T$ for flat currents $T \in \cF_k(E)$ in such a way that \eqref{eq: homotopy-boundary} still holds. It is not difficult to show that for $T,S\in \cF_k(E)$ with $T = S \mod p$ we have
        \begin{equation}\label{eq: product-current-mod-p}
            \bb{t}\times T = \bb{t}\times S \mod p \quad \textup{and} \quad \bb{0,1}\times T= \bb{0,1}\times S \mod p.
        \end{equation}
        Let $H \colon [0,1] \times E \to X$ be Lipschitz such that for a fixed $t \in [0,1]$ the map $x \mapsto H(t,x)$ is $L$-Lipschitz and for a fixed $x\in E$ the map $t \mapsto H(t,x)$ is $\gamma$-Lipschitz. A computation exactly as in the proof of \cite[Proposition 2.10]{wenger-isop} shows that 
        \begin{equation}\label{eq: mass-homotopy}
            \bM(H_\#(\bb{0,1}\times T)) \leq (k+1) \gamma L^k \bM(T),
        \end{equation}
        for all $T \in \bI_k(E)$. It follows that if $T \in \cF_k(E)$ has finite $p$-mass $\bM_p$, then \eqref{eq: mass-homotopy} holds for $T$ with $\bM$ replaced by $\bM_p$. Using this, it is not difficult to show that 
        \begin{equation}\label{eq: flat-homotopy}
            \cF_p(H_\#(\bb{0,1}\times T)) \leq  (k+1) \gamma L^k \cF_p(T)
        \end{equation}
        for all $T \in \cF_k(E)$ with finite $p$-mass $\bM_p$. Next, we explain the cone construction. We refer to \cite{wenger-isop} for more details. Fix some point $x\in E$. We denote by $H\colon [0,1]\times E \to E$ the straight-line homotopy between the identity on $E$ and the constant map equal to $x$. For $T \in \bI_k(E)$, the cone over $T$ with vertex $x$ is defined as $\{x\} \times T \coloneqq H_\#(\bb{0,1}\times T) \in \bI_{k+1}(E)$. By \eqref{eq: homotopy-boundary} and \eqref{eq: mass-homotopy} it has the following properties
        $$\partial (\{x\} \times T) =  T - \{x\} \times \partial T \quad\textup{and}\quad \bM_p(\{x\}\times T) \leq 2r \bM_p(T),$$
        where $r>0$ is the radius of the smallest closed ball that contains the support of $T$. Moreover, if $T,R\in \bI_k(E)$ and $S\in \bI_{k+1}(E)$ are such that $T = R + \partial S$, then
        $$\{x\}\times T = \{x\}\times (R +\partial S) = \{x\}\times (R+S) + \partial(\{x\}\times S).$$
        Therefore, $\cF(\{x\}\times T) \leq 2 \diam (E) \cF(T)$. This leads to
        $$\cF_p(\{x\}\times T) \leq 2\diam( E) \cF_p(T)$$
        for all $T \in \cF_k(E)$. In particular, if a sequence $T_i \in \cF_k(E)$ satisfies $\cF_p(T_i-T) \to 0$ for some $T \in \cF_k(E)$, then also the cones converge $\cF_p(\{x\}\times T_i -\{x\}\times T) \to 0$.

    \subsection{Compactness theorem for integral currents mod $p$}
       We essentially follow the proof of the compactness theorem in \cite{Adams-chains}. The definition of flat chains used in \cite{Adams-chains} differs from the definition of flat currents used in this article. Thus, we provide the proof of the compactness theorem.

    \begin{theorem}\label{thme: compactness-flat-mod-p}
        Let $E$ be a compact and convex subset of $l^\infty$. Suppose that $T_i\in \bI_{k}(E)$ is a sequence satisfying
        $$\sup_i \bM_p(T_i) + \bM_p(\partial T_i) < \infty.$$
        Then, there exists a subsequence $\big(T_{i(j)}\big)_{j \in \N}$ of $(T_i)_{i\in \N}$ and $T\in \cF_k(E)$ such that $\cF_p(T_{i(j)}-T) \to 0$ as $j \to \infty$.
    \end{theorem}

    \begin{proof}
        Define $C = \sup_i \bM_p(T_i) + \bM_p(\partial T_i)$. We argue by induction over the dimension $k$. Let $k =0$. Put $K = \{ T \in \cF_k(E) \colon \bM_p(T) + \bM_p(\partial T) \leq C\}$. It follows from the lower semicontinuity of the $p$-mass that $K$ is closed. Given $\epsilon>0$, we cover $E$ by a finite family of balls $\{B(x_i,\epsilon)\}_{i=1}^N$ and write $Q \subset K$ for the subset of flat currents of the form 
        $$\sum_i^N q_i \bb{x_i}, \quad |q_i|<p \quad \textup{and} \quad \sum_i^N |q_i| \leq C.$$
        Here, $\bb{x}$ denotes the $0$-current defined by $\bb{x}(f) = f(x)$ for all $f\in \LIP_b(X)$. Since $E$ is convex, there exists $S \in Q$ for every $T \in K$ satisfying $\cF_p(T-S) \leq C\epsilon$. This shows that $K$ is totally bounded and hence, compact.

        \medskip

        We pass to the induction step. Let $x\in E$ and $r>0$. We claim that for almost all $\delta \in (r,2r)$, there exists a subsequence $T_{i(j)}$ such that for every $j \in \N$ we have $T_{i(j)}\on B(x,\delta) \in \bI_k(E)$ and 
        \begin{equation*}\label{eq: compactnes-loc}
            \bM_p(T_{i(j)}\on B(x,\delta) ) + \bM_p(\partial( T_{i(j)}\on B(x,\delta)) ) \leq \frac{2C}{\delta} + C.
        \end{equation*}
        Indeed, define $\pi\colon E \to \R, \;y \to d(x,y)$, then Lemma \ref{lemma: slicing-mod-p} and Fatou's lemma imply
        $$\int_\delta^{2\delta} \liminf_{i \to \infty} \bM_p(\langle T_i,\pi,t\rangle )\;dt\leq \liminf_{i \to \infty}  \int_\delta^{2\delta}\bM_p(\langle T_i,\pi,t\rangle ) \;dt\leq \sup_i \bM_p(T_i) \leq C. $$
        Therefore, for almost every $\delta\in (r,2r)$, we have $ \liminf_i\bM_p(\langle T_i,\pi,\delta\rangle ) \leq C/\delta$. In particular, we can find a subsequence $T_{i(j)}$ satisfying $\bM_p(\langle T_{i(j)},\pi,\delta\rangle ) \leq 2C/\delta$ and $\langle T_{i(j)},\pi,\delta\rangle\in \bI_{k-1}(E)$ for every $j\in \N$. By the definition of the slice operator we have
        \begin{align*}
            \bM_p(\partial (T_{i(j)}\on B(x,\delta) )) \leq  \bM_p(\langle T_{i(j)},\pi,\delta\rangle ) + \bM_p((\partial T_{i(j)})\on B(x,\delta))
        \end{align*}
        and $T_{i(j)}\on B(x,\delta) \in \bI_k(E)$ for every $j \in \N$. Since $\sup_i \bM_p(T_i) + \bM_p(\partial T_i) \leq C$, the claim follows. Now, let $\eta>0$ be such that $ (1 + 8C)\eta < \epsilon$. Applying the claim repeatedly, we find a finite cover of $E$ that consists of balls $\{B_j\}_j^n$ with radii in $(\eta,2\eta)$ and a subsequence of $(T_i)_i$ (still denoted by $(T_i)_i$) such that $T_i\on B_j \in \bI_k(E)$ and
        \begin{equation}\label{eq: compactnes-loc-new}
            \bM_p(T_{i}\on B_j ) + \bM_p(\partial( T_{i}\on B_j) ) \leq \frac{2C}{\delta} + C.
        \end{equation}
        for all $i,j \in \N$. We decompose $E \setminus \bigcup_{j=1}^N \partial B_j$ into a finite partition $\{U_l\}_{l=1}^N$. For $i \in \N$ and $l=1,\dots, N$, we write $T_i^l = T_i \on U_l$. Lemma \ref{lemma: mass-equal-mod-mass-for-reduction} implies that the mass $p$-measure of each $T_i$ is additive and hence,
        $$\bM_p(T_i) = \sum_l^N \bM_p(T_i^l),$$
        for every $i \in \N$. By \eqref{eq: compactnes-loc-new} we have $\sup_i \bM_p(\partial T_i^l) < \infty$ for each $l=1,\dots, N$. Thus, we conclude from the induction hypothesis that we may pass to another subsequence (we keep writing $(T_i)_i$) such that $\partial T_i^l$ is Cauchy with respect to $\cF_p$ for every $l = 1,\dots, N$. We fix a point $x_l\in U_l$ for each $U_l$. It follows that every sequence $\{x_l\}\times \partial T_i^l$ is Cauchy as well. Therefore, there exists $I\in \N$ such that 
        $$\cF_p(\{x_l\}\times \partial T_i^l-\{x_l\}\times\partial  T_{i'}^l) < \frac{\eta}{N}$$
        for every $l =1,\dots, N$ and all $i,i'\geq I$. Fix two integers $a,b$ greater than $I$ for the moment. Then, for each $l=1,\dots,N$, there exist $R_l \in \cI_{k}(E)$, $S_l \in \cI_{k+1}(E)$ and $Q_l \in \cF_k(E)$ satisfying $\{x_l\}\times \partial T_a^l-\{x_l\}\times\partial  T_{b}^l =  R_l + \partial  S_l +pQ_l$ and
        $$\bM_p(S_l) + \bM_p(R_l) \leq \frac{\eta}{N}.$$
        Therefore, using that $\partial( \{x\}\times T) = T -\{x\}\times \partial T$ for each $T \in \bI_k(E)$, we conclude
        \begin{align*}
            (T_a^l -T_b^l) &= \partial (\{x_l\}\times  T_a^l-\{x_l\}\times  T_{b}^l) -  R_l + \partial  S_l +pQ_l.
        \end{align*}
        Since each $U_l$ has radius at most $2\eta$ we have 
        $$\bM_p(\{x_l\}\times  T_a^l-\{x_l\}\times  T_{b}^l) \leq 4\eta( \bM_p(T_a^l)+\bM_p(T_b^l)),$$
        for every $l=1,\dots,N$. It follows
        \begin{align*}
            \cF_p(T_a^l-T_b^l) &\leq  \bM_p(\{x_l\}\times  T_a^l-\{x_l\}\times  T_{b}^l)+\bM_p(S_l) + \bM_p(R_l)
            \\
            &\leq 4\eta( \bM_p(T_a^l)+\bM_p(T_b^l)) + \frac{\eta}{N}.
        \end{align*}
        Finally, 
        \begin{align*}
            \cF_p(T_a-T_b) &\leq \sum_l^N \cF_p(T_a^l-T_b^l)
            \\
            &\leq \sum_l^N 4\eta( \bM_p(T_a^l)+\bM_p(T_b^l)) + \frac{\eta}{N}
            \\
            &\leq 4\eta\sum_l^N ( \bM_p(T_a^l)+\bM_p(T_b^l))+ \eta
            \\
            &\leq 4\eta( \bM(T_a) + \bM(T_b))+\eta
            \\
            &\leq (1 + 8C)\eta < \epsilon.
        \end{align*}

        Since $\epsilon>0$ was arbitrary this completes the proof.

        \end{proof}

    \subsection{Rectifiability of flat currents}
        Let $E$ be a compact convex subspace of a Banach space. It was shown in \cite{Ambrosio-Wenger} that the $p$-mass $\norm{T}_p$ of a flat $k$-current $T$ in $E$ with finite $p$-mass is concentrated on a $k$-rectifiable set if the Banach space satisfies a strong finite-dimensional approximation property. We are able to relax the condition on the ambient Banach space but we need to assume stronger conditions for $\norm{T}_p$. We note that for flat chains with finite mass there is an even more general rectifiability result available \cite[Corollary 8.2.3]{De-pauw-rect-flat-chains}. However, in order to use this result, we need to connect the theories of flat currents and flat chains. In addition, the proof of the rectifiability result presented here uses a different approach which in the opinion of the author is of independent interest. 

        \medskip

        We need the following lemma.

        \begin{lemma}\label{lemma: map-modulo-2}
            Let $T\in \cF_k(l^\infty)$. Suppose that $T$ has finite $p$-mass and $\spt_pT$ is compact. Then, for every $\epsilon>0$ there exist a finite dimensional linear subspace $V\subset l^\infty$ and a $1$-Lipschitz map $\pi\colon l^\infty\to V$ such that $\cF_p(T-\pi_\#T) \leq 2(k+1)\epsilon \cF_p(T)$.
        \end{lemma}

        \begin{proof}
            Let $\epsilon>0$. Set $K = \spt_p T$. Since $l^\infty$ has the metric approximation property, there exist a finite dimensional linear subspace $V\subset l^\infty$ and a $1$-Lipschitz map $\pi\colon l^\infty\to V$ such that $\norm{\pi(x)-x}_{l^\infty}<\epsilon$ for each $x \in K$; see e.g. \cite[Proposition A.6]{de-pauw-map}. Let $H\colon [0,1]\times K \to V$ be the straight line homotopy between the inclusion $\iota_K$ of $K$ into $l^\infty$ and $\pi$. By \eqref{eq: homotopy-boundary} we have 
            $$\partial ( H_\#(\bb{0,1}\times T))- H_\#(\bb{0,1}\times \partial T)  = T - \pi_\# T.$$
            Notice that for a fixed $t$ the map $x \mapsto H(t,x)$ is $1$-Lipschitz and for a fixed $x \in K$ the map $t \mapsto H(t,x)$ is $\epsilon$-Lipschitz. Therefore, \eqref{eq: flat-homotopy} implies
            \begin{align*}
                \cF_p(T-\pi_\#T)  \leq  \cF_p(H_\#(\bb{0,1}\times T)) + \cF_p(H_\#(\bb{0,1}\times \partial T))\leq 2(k+1)\epsilon \cF_p(T).
            \end{align*}
            This completes the proof.  
    \end{proof}

    We can now prove the following rectifiability result for flat currents with finite $p$-mass.

    \begin{proposition}\label{prop: rectifiability-flat-current}
        Let $E$ be a compact and convex subset of $l^\infty$ and let $T\in \cF_k(E)$. Suppose that $T$ has finite $p$-mass and $\norm{T}_p$ is concentrated on a set with finite Hausdorff $k$-measure. Then, $\norm{T}_p$ is concentrated on a $k$-rectifiable set $A \subset E$. Moreover, there exists $S\in \cI_k(E)$ with $S = T \mod p$.
    \end{proposition}

    \begin{proof}
        Let $C \subset E$ be the set with finite Hausdorff $k$-measure $\norm{T}_p$ is concentrated on, and let $P \subset C$ be purely $k$-unrectifiable. We need to prove that $\norm{T}_p(P)= 0$. Since $\norm{T}_p$ is a Borel measure, it follows from inner regularity that it suffices to show that $\norm{T}_p(P')=0$ for all closed subsets $P' \subset P$ in order to show that $\norm{T}_p(P)=0$. Therefore, we may suppose that $P$ is closed. We claim that $\varphi_\#  T\on P = 0 \mod p$ for all Lipschitz maps $\varphi\colon E \to V$ whenever $V$ is a finite dimensional Banach space. Clearly, the statement is invariant under bi-Lipschitz changes of the metric on $V$. Hence, it suffices to prove the claim with the additional assumption that $V$ is equal to some $\R^N$ equipped with the standard Euclidean distance. Let $\varphi \colon E \to \R^N$ be $L$-Lipschitz and let $\epsilon>0$. It follows from Theorem \ref{thme: bate-perturb} that there exists a $L$-Lipschitz map $\psi \colon E \to \R^N$ satisfying $d(\varphi,\psi) < \epsilon$ and $\Ha^k(\psi(P))=0$. Notice that
        $$\spt_p(\psi_\#(T \on P)) \subset \psi (\spt_p(T\on P)) \subset \psi(P). $$
        Furthermore, $\psi_\#(T \on P)$ is a flat $k$-current in $\R^N$ with finite $p$-mass $\bM_p$. By \cite[Corollary 1.3]{Ambrosio-Wenger}, we have that $\norm{\psi_\#(T \on P)}_p$ is absolutely continuous with respect to the Hausdorff $k$-measure $\Ha^k$ and hence, $\psi_\#(T \on P) = 0\mod p$. Let $H \colon [0,1] \times E \to \R^N$ be the straight line homotopy between $\varphi$ and $\psi$. By \eqref{eq: homotopy-boundary} we have 
        $$\partial (H_\#(\bb{0,1}\times T \on P))+  H_\#(\bb{0,1}\times \partial ( T \on P)) = \psi_\# T \on P - \varphi_\# T \on P.$$ 
        For a fixed $t$, the map $x \mapsto H(t,x)$ is $L$-Lipschitz and for a fixed $x \in K$, the map $t \mapsto H(t,x)$ is $\epsilon$-Lipschitz. It follows from \eqref{eq: flat-homotopy} that
        $$\cF_p( H_\#(\bb{0,1}\times T \on P)) \leq  (k+1)\epsilon L^k \cF_p(T \on P)$$
        and
        $$\cF_p( H_\#(\bb{0,1}\times \partial (T \on P))) \leq  (k+1)\epsilon L^{k-1} \cF_p(\partial (T \on P)).$$
        Therefore,
        $$\cF_p(\varphi_\#(T\on P)) =\cF_p(\varphi_\#(T\on P)-\psi_\#(T\on P)) \leq (k+1)\epsilon (L^{k-1}+L^k) \cF_p(T\on P).$$
        Since $\epsilon>0$ was arbitrary, we conclude that $\varphi_\#  T\on P = 0\mod p$. This proves the claim. Lemma \ref{lemma: map-modulo-2} implies that there exists a sequence of finite dimensional vector subspaces $V_i \subset l^\infty$ and Lipschitz maps $\pi_i\colon E \to V_i$ such that $\cF_p(T\on P -\pi_{i\#}(T\on P) ) \leq (k+1)\epsilon \cF_p(T\on P )$. It follows from the claim that $\pi_{i\#}(T\on P) = 0 \mod p$ for all $i \in \N$. By passing to the limit we get $T \on P=0 \mod p$ and in particular, $\norm{T}_p(P) = 0$. We conclude that $\norm{T}_p$ is concentrated on a $k$-rectifiable set. The second part of the statement can now be proven exactly as \cite[Corollary 1.3]{Ambrosio-Wenger}.
    \end{proof}

\section{Linear local contractibility}\label{sec: llc}
    Throughout this section, $X$ denotes a metric space. Let $R>0$ and $\lambda \geq 1$. A subset $B \subset X$ is called $(\lambda,R)$-linearly locally contractible if for every $x \in B$ and every $r \in (0,R)$, the ball $B(x,r)$ is contractible within $B(x,\lambda r)$. We say $X$ is $(\lambda,R)$-linearly locally contractible around a point $x \in X$ if the ball $B(x,R)$ is $(\lambda,R)$-linearly locally contractible. Finally, suppose that $X$ has a finite Hausdorff $n$-measure. Then, $X$ is said to be almost everywhere linearly locally contractible if $X$ is linearly locally contractible around almost all points $x \in X$. Notice that in the previous definition, we do allow the linear local contractibility constants to depend on the point. The following easy example shows that this definition is indeed weaker than the standard definition. Let $C$ be a hyperbolic cone that becomes thinner towards the tip $y$ of the cone. Then $C$ is $(1, R)$-linearly locally contractible around each $x\neq y$, provided that $R$ is chosen sufficiently small such that $B(x, 2R)$ does not wrap around the cone. However, $C$ is not $(\lambda,R)$-linearly locally contractible around the tip $y$. 

    \medskip

    Linear local contractibility is a useful tool for constructing and extending continuous functions in a controlled manner. Here, we extend these classical techniques using the weaker version of linear local contractibility that was introduced above. We need the following definition. A map $f\colon X\to Y$ between two metric spaces is called $\varepsilon$-continuous if there exists $\delta>0$ such that $d(f(x), f(y))<\varepsilon$ for all $x, y\in X$ with $d(x,y)<\delta$.

    \begin{proposition}\label{prop:eps-cont-extension}
        Let $X$, $Y$ be compact metric spaces such that $X$ has topological dimension at most $n$ and $B\subset Y$ is a closed subset that is $(\lambda,R)$-linearly locally contractible. Let $A$ be a (possibly empty) subset of $X$. There exists $Q\geq 1$ such that if $f\colon X \to Y$ is $\varepsilon$-continuous on $X$ with $\varepsilon<R/Q$ and $f$ is continuous at every point of $A$ and $f(X\setminus A)\subset B$, then there is a continuous map $g\colon X\to Y$ which coincides with $f$ on $A$ and satisfies $d(f,g)<Q\cdot\varepsilon$.
        Here, $Q$ depends only on $\lambda$ and $n$.
    \end{proposition}
    
    We record the following consequence of the proposition: if two maps $f_0, f_1\colon X\to B$ are continuous and satisfy $d(f_0,f_1)<\varepsilon$, then they are homotopic through a homotopy $H\colon [0,1]\times X\to Y$ satisfying 
    $$
    d(H(t,x), f_0(x))<(1+Q)\cdot\varepsilon
    $$ 
    for every $x\in X$ and each $t\in[0,1]$.
    \medskip
    
    The proposition can be proven analogously to the corresponding result in \cite{Petersen93}; see also \cite[Proposition 5.4]{semmes-curves}. We therefore only provide a sketch of the proof. 
    
    \begin{proof}
        Let $Q\geq 1$ be sufficiently large to be determined later and let $\delta>0$ be such that $d(f(x),f(y))< \epsilon<R/Q$ whenever $d(x,y)<\delta$. Since $X$ is compact and has topological dimension at most $n$, there exists an open covering $\left\{U_i\right\}_{i\in \N}$ of $X\setminus A$ with multiplicity at most $n+1$ and 
        \begin{equation}\label{eq: llce-ext-control-diam-cover}
            \diam U_i < \frac{1}{2} \min\{\delta,d(U_i,A)\}
        \end{equation}
        for all $i \in \N$. Furthermore, there exist a simplicial complex $\Sigma$ of dimension $\leq n$ and a continuous map $\varphi\colon X \setminus A \to \Sigma$ with the following property. Whenever $x \in U_i$ and $y \in U_j$, then $\varphi(x)$ and $ \varphi(y)$ are contained in a common simplex if and only if $U_i \cap U_j \neq \emptyset$. It follows that for each $i \in \N$, there is a vertex $v_i \in \Sigma$ such that $\varphi(U_i)$ is contained in a union of simplices with $v_i$ as a vertex. The simplicial complex $\Sigma$ is usually referred to as the nerve of the cover $\left\{U_i\right\}_{i\in \N}$; see e.g. \cite[Section 3]{Petersen93}. Let $\Sigma^0 = \{v_1,v_2,\dots\}$ be the $0$-skeleton of $\Sigma$, where each $v_i$ corresponds to $U_i$ as explained before. For each $i \in \N$, fix a point $x_i \in U_i$ and define $\psi^0\colon \Sigma^0 \to B$ by $\psi^0(v_i)=f(x_i)$ for each vertex $v_i$. Let $v_i,v_j$ be two vertices of the same simplex. By \eqref{eq: llce-ext-control-diam-cover}, we have $d(x_i,x_j) < \delta$ and hence, $d(\psi^0(v_i),\psi^0(v_j)) <\epsilon$. Therefore, if $Q>0$ is sufficiently large with respect to $\lambda$, $n$ and $R$ we can use the linear local contractibility to extend $\psi^0$ inductively over the skeleton of $\Sigma$ to a continuous map $\psi\colon \Sigma\to Y$ such that the following holds. For each $i\in \N$ and every $y \in \Sigma$ that is contained in a simplex with $v_i$ as a vertex, we have
        \begin{equation}\label{eq: llc-ext-control-extension}
            d(g(y),g(x_i)) \leq C \diam U_i,
        \end{equation}
        where $C$ depends only on $n$ and $\lambda$. Finally, we define $g \colon X \to Y$ as $g=f$ on $A$ and as $\psi\circ\varphi  $ on $X \setminus A$. Using \eqref{eq: llce-ext-control-diam-cover} and \eqref{eq: llc-ext-control-extension}, it is not difficult to show that $g$ is continuous and satisfies $d(f,g)< Q \epsilon$.
        \end{proof}

    \subsection{Linear local contractibility and manifolds}\label{sec: llc-mfld}
         For the remainder of the section, suppose that $X$ has finite Hausdorff $n$-measure and is homeomorphic to a compact smooth $n$-manifold $M$. We denote the homeomorphism by $\varrho\colon M\to X $.
         \medskip
         
         First, we describe the construction of the orientable double cover of a manifold without boundary. Recall that for each $x\in X$, the relative homology group $H_n(X, X\setminus x)$ is infinite cyclic. The orientable double cover of $X$ is defined as 
         $$\tilde{X} = \big\{ (x,o_x) \colon x\in X,\; o_x \textup{ is a generator of }H_n(X, X\setminus x)\big\}.$$
         For each $x \in X$, there exist two points $(x,o_x^+),(x,o_x^-) \in \tilde{X}$ that correspond to the positive and negative local orientations at $x$. Furthermore, there exists a topology on $\tilde{X}$ such that the map $\pi\colon \tilde{X}\to X, (x,o_x) \mapsto x$ is a (topological) double cover. Therefore, for each $x\in X$ there exists an open connected neighborhood $U\subset X$ such that $\pi^{-1}(U)$ consists of two components $U^+,U^-\subset\tilde{X}$ and $\pi$ restricted to either of these components is a homeomorphism. The orientable double cover is a closed, orientable manifold and is connected if and only if $X$ is non-orientable. Finally, we describe how to construct a metric on $\tilde{X}$. Let $\{U_i\}_{i=1}^N$ be an open cover of $X$ such that for each $i = 1,\dots, N$, the preimage $\pi^{-1}(U_i)$ consists of two components that are homeomorphic to $U_i$. We endow each $U_i^+$ and $U_i^-$ with a metric such that $\pi$ restricts to an isometry on each member of $\{U_i^+,U_i^+\}_{i=1}^N$. Notice that if two sets $V,W\in \{U_i^+,U_i^+\}_i^N$ have non-empty intersection, then for all $x,y\in V\cap W$, the distance between $x$ and $y$ in $V$ and in $W$ is equal. It follows that there exists a metric on $\tilde{X}$ that agrees with the topology, such that the map $\pi:\tilde{X} \to X$ is a local isometry. The next lemma is an immediate consequence of the definitions and the fact that $\pi$ is a local isometry. 

         \begin{lemma}\label{lemma: ae-llc-orientable-double-ae-lcc}
            If $X$ is linearly contractible around $x$, then $\tilde{X}$ is linearly contractible around $(x,o_x^+)$ and $(x,o_x^+)$. In particular, if $X$ is (almost everywhere) linearly locally contractible, then $\tilde{X}$ is (almost everywhere) linearly locally contractible. 
         \end{lemma}

         We can now easily prove the validity of the relative isoperimetric inequality in non-orientable metric manifolds that are Ahlfors $n$-regular and linearly locally contractible. 

        \begin{proof}[Proof of Theorem \ref{thme: poincare}]
            Let $X$ be a metric space that is homeomorphic to a closed, non-orientable smooth $n$-manifold. Furthermore, suppose that $X$ is Ahlfors $n$-regular and linearly locally contractible. Let $\tilde{X}$ be the orientable double cover of $X$ and $\pi\colon \tilde{X}\to X$ be the covering map. By Lemma \ref{lemma: ae-llc-orientable-double-ae-lcc} the orientable double cover $\tilde{X}$ is linearly locally contractible as well. Moreover, the covering map $\pi$ is a local isometry and thus, $\tilde{X}$ is also Ahlfors $n$-regular. It follows from \cite[Theorem 1.5]{BMW} that $\tilde{X}$ satisfies a relative isoperimetric inequality. Therefore, using again that $\pi$ is a local isometry, we conclude that $X$ satisfies a local relative isoperimetric inequality. That is,
            $$\min\bigl\{\Ha^n\big(E\cap B\big), \Ha^n\big(B \setminus E\big) \bigr\} \leq C \cdot\mathscr{M}_-\big( E\, | \, \lambda B\big)^{\,\tfrac{n}{n-1}}$$
            holds for every Borel subset \(E\subset X\) and every open ball \(B\subset X\) with sufficiently small radius. Since $X$ is compact the local relative isoperimetric inequality upgrades to a global inequality. Notice the Ahlfors regularity and linear contractibility constants of $X$ and $\tilde{X}$ are comparable. Hence, the constants $C,\lambda\geq 1$ depend only on the data of $X$. This completes the proof.
        \end{proof}

         Next, we explain how to glue two copies of $X$ along their boundaries.

        \begin{definition}\label{def: gluing-boundary}
            Suppose that $X$ has boundary. Let $\hat{X}$ be the space obtained as the quotient $\hat{X}= X_1 \sqcup X_2 / \sim$, where $X_1$ and $X_2$ are two copies of $X$ and $x\sim y$ if and only if $x\in \partial X_1$, $y\in \partial X_2$ and $x=y$. We equip $\hat{X}$ with the quotient metric $\hat{d}$, which in this case has the following form
            $$
            \hat{d}(x,y)=
            \begin{cases}
            d(x,y)\; &\textup{if } x,y\in X_1\textup{ or }x,y\in X_2,\\
            \inf_{z \in \partial X} d(x,z)+d(z,y) &\textup{if } x\in X_1, y\in X_2\textup{ or }y\in X_1, x\in X_2.
            \end{cases}
            $$
        \end{definition}

        It is not difficult to show that $\hat{X}$ is a closed manifold and orientable in case $X$ is orientable. We write $\iota_1,\iota_2\colon X\to \hat{X}$ for the isometric embeddings of $X$ into $\hat{X}$ that identify $X$ with $X_1$ and $X_2$, respectively. Given a map $\varphi\colon X \to M$ satisfying $\varphi(\partial X)\subset \partial M$, we let $\hat{\varphi}\colon \hat{X}\to \hat{M}$ be the map such that $\hat{\varphi}(\iota_i^X(x))=\iota_i^M(\varphi(x))$ for every $x\in X$ and $i = 0,1$.


        \[
        \begin{tikzcd}[row sep=huge, column sep=huge,
        nodes={scale=1.3, text depth=0pt},nodes={inner xsep=0.32em}]
            M_1
                \arrow[r,hookrightarrow, "\iota_1^M"]
            & \hat{M}
                \arrow[r,hookleftarrow, "\iota_2^M"]
            & M_2
            \\
            X_1
                \arrow[r,hookrightarrow, "\iota_1^X"]                
                \arrow[u,"\varphi"]
            & \hat{X}
                \arrow[u,"\hat{\varphi}"]
                \arrow[r,hookleftarrow, "\iota_2^X"]
            & X_2
                \arrow[u,"\varphi"]
        \end{tikzcd}
        \]

        It is not difficult to show that if $\varphi$ is continuous or Lipschitz, then $\hat{\varphi}$ is continuous or Lipschitz as well. In particular, $\hat{X}$ is homeomorphic to $\hat{M}$ via the homeomorphism $\hat{\varrho}\colon \hat{M}\to \hat{X}$. Furthermore, if $X$ is orientable, then the degree $\deg(\varphi)$ of $\varphi$ is well-defined and we have $\deg(\varphi)=\deg(\hat{\varphi})$. Finally, we record the following lemma for later use.

        \begin{lemma}\label{lemma: ae-llc-bdry-double-ae-llc}
            If $X$ is $(\lambda,R)$-linearly contractible around $x\in X\setminus \partial X$, then $\hat{X}$ is $(\lambda,\hat{R})$-linearly contractible around $\iota_1(x)$ and $\iota_2(x)$, where $\hat{R}= \min\{R, d(x,\partial X)/3\lambda\}$. In particular, if $\Ha^n(\partial X) = 0$ and $X$ is almost everywhere linearly locally contractible, then $\hat{X}$ is almost everywhere linearly locally contractible. 
        \end{lemma}

        \begin{proof}
            Put $\hat{x} = \iota_1(x)$. Let $\hat{y}\in B(\hat{x},\hat{R})$ and $r\in (0, \hat{R})$. By assumption, the ball $B(y,r)$ is contractible within $B(y,\lambda r)$, where $y = \iota_1^{-1}(\hat{y})$. Notice that $B(y,\lambda r) \subset B(x,2\lambda \hat{R})$ and $ B(x,2\lambda \hat{R}) \cap \partial X = \emptyset$. Therefore, $\iota_1$ restricted to $B(x,2\lambda \hat{R})$ is an isometry and in particular, $B(\hat{y},r)$ is contractible within $B(\hat{y},\lambda r)$. The proof for $\iota_2(x)$ is analogous. Finally, since $\Ha^n(\iota_1(\partial X)) = 0$ the second part of the statement follows directly from the first part. 
        \end{proof}

        We conclude the discussion with the following lemma, which will be useful to prove the main theorems.

        \begin{lemma}\label{lemma: fund-class-closed-gives-fund-class-bdry}
            Suppose that $\hat{X}$ has a metric fundamental class $\hat{T} \in \cI_n(\hat{X})$ (modulo $2$). Then, $T =\iota^{-1}_{1\#}(\hat{T}\on \iota_1(M_1)) \in \cI_n(X)$ defines a metric fundamental class (modulo $2$) of $X$. Moreover, suppose that $\hat{T}$ has the following uniqueness property: whenever $S\in \bI_n(\hat{X})$ is an $n$-cycle (modulo $2$), then there exists $k \in \Z$ (or $k=0,1$) such that $S= k\cdot \hat{T}$ (or $S= k\cdot \hat{T}\mod 2$). Then, $T$ satisfies the same uniqueness property for integral currents $S \in \bI_n(X)$ with $\spt \partial S \subset \partial X$ or ($\spt_2 \partial S \subset \partial X$).
        \end{lemma}

        \begin{proof}
            Let $\hat{T}$ be a metric fundamental class of $\hat{X}$. Then there exists $C>0$ such that $\norm{\hat{T}}\leq C \cdot \Ha^n_{\hat{X}}.$ It follows that $T$ satisfies the same inequality because $\iota_1^X$ is an isometric embedding. Now, let $\varphi \colon X \to M$ be Lipschitz such that $\varphi(\partial X) \subset \partial M$. Then there exists a Lipschitz map $\hat{\varphi}\colon \hat{X}\to \hat{M}$ satisfying $\deg(\varphi) = \deg(\hat{\varphi})$ and $\hat{\varphi}(\iota_i^X(x))=\iota_i^M(\varphi(x))$ for every $x\in X$ and $i = 0,1$. Therefore,
            \begin{align*}
                \varphi_\# T& = ((\iota^{M}_1)^{-1} \circ \hat{\varphi}\circ \iota_1^X)_{\#} T =((\iota^{M}_1)^{-1}  \circ \hat{\varphi})_\# (\hat{T}\on \iota_1^X(X)) 
                \\
                &= (\iota^{M}_{1\# })^{-1}  (\deg(\hat{\varphi}) \cdot(\bb{\hat{M}}\on \iota_1^M(M))) = \deg(\varphi)\cdot \bb{M}.
            \end{align*}
            Finally, let $S \in \bI_n(X)$ be such that $\spt (\partial S) \subset \partial X$. Then $\hat{S} = \iota_{1\#}^X S - \iota_{2\#}^X S $ defines an integral $n$-cycle in $\hat{X}$. Hence, there exists $k \in \Z$ such that $k \cdot \hat{T} = \hat{S}$ and in particular, 
            $$k \cdot T = k \cdot\iota^{-1}_{1\#}(\hat{T} \on \iota_1^X(X)) = \iota^{-1}_{1\#}(\hat{S} \on \iota_1^X(X)) = S.$$
            In case $\hat{T}$ is a metric fundamental class modulo $2$, the proof is analogous.
        \end{proof}

    \subsection{Linear local contractibility and degree}
        Let $X$ be a closed metric $n$-manifold with finite Hausdorff $n$-measure. The next lemma provides a continuous extension of bi-Lipschitz maps from Euclidean space into $X$ that are compatible with the topology of $X$.

    
        \begin{lemma}\label{lemma:bi-lip-ext-LLC}
            Let $K\subset\R^n$ be compact and $\varphi\colon K\to X$ be bi-Lipschitz. Suppose that the image of $\varphi$ is contained in a linearly locally contractible subset $B\subset X$. Then, there exists a continuous extension $\overline{\varphi}\colon U\to X$ of $\varphi$ to an open neighborhood $U$ of $K$ and $C>0$ such that
    
            \begin{equation}\label{eq:bi-lip-ext-LLC}
                d(\varphi(x),\overline{\varphi}(y)) \leq C \norm{x-y}
            \end{equation}
            for each $x \in K$ and every $y \in U$. Here, $C>0$ depends only on $n$, the bi-Lipschitz constant of $\varphi$ and the linear local contractibility constants of $B$.
        \end{lemma}

         The lemma is a localized version of \cite[Lemma 5.3]{BMW} and is based on the same argument. 

        \begin{proof}
            Let $\{Q_i\}_i$ be a Whitney cube decomposition of $\R^n\setminus K$, see \cite[Theorem VI.1.1]{Stein-singular-integrals-1970}. That is, the cubes $\{Q_i\}_i$ cover $\R^n\setminus K$, have disjoint interiors and 
            \begin{equation}\label{eq: whitney cube}
                n \diam Q_i \leq d(Q_i,K) \leq 4n\diam Q_i
            \end{equation}
            for each $i \in \N$. Let $\delta>0$ be sufficiently small to be determined later. We define $U$ as the union of all cubes $Q_i$ that intersect the $\delta$-neighborhood. Let $U^0 = \{v_1,v_2,\dots\}$ be the set of all vertices of $U$. For each $l \in \N$, let $x_l \in K$ be such that  $d(v_l,x_l)\leq 2 d(Q_i,K)$, where $Q_i$ is a cube with $v_l$ as a vertex. We define $\overline{\varphi} \colon U^0\to X$ by $\overline{\varphi}(v_l) = \varphi(x_l)$. Let $v_l,v_k\in Q_i$ be two vertices and let $L$ be the bi-Lipschitz constant of $\varphi$. Then
            $$d(\overline{\varphi}(v_l),\overline{\varphi}(v_k)) \leq L \norm{x_l-x_k} \leq L (\diam  Q_i+ 4d(Q_i,K)) \leq 5 L d(Q_i,K) < 5L\delta.$$
            Therefore, if $\delta$ is sufficiently small, we can use that $\overline{\varphi}(U^0)$ is contained in a linearly locally contractible subset $B \subset X$ to extend $\overline{\varphi}$ to a continuous map on $U$ such that $\diam \overline{\varphi}(Q_i) \leq c d(Q_i,K)$ holds for each cube $Q_i\subset U$. Here, $c$ depends only on $n$, $L$ and the linear local contractibility constants of $B$. Finally, let $x \in K$ and $y \in U$. Choose a cube $Q_i\subset U$ containing $y$, and let $v_l$ be a vertex of $Q_i$. Then
            $$d(\varphi(x),\overline{\varphi}(y)) \leq L \norm{x-x_l} + c \diam Q_i\leq (4L + c) d(Q_i,K) \leq  (4L + c) \norm{x-y}.$$
            This completes the proof. 
        \end{proof}

        We record the following consequence of the previous lemma. 
        
        \begin{lemma}\label{lemma: bi-lip-ext-llc-degree}
            Let $\overline{\varphi}\colon U \to X$ be a continuous extension of a bi-Lipschitz map $\varphi\colon K\to X$ as in Lemma \ref{lemma:bi-lip-ext-LLC}. Furthermore, let $\pi\colon X\to \R^n$ be Lipschitz and let $\psi\colon \R^n \to \R^n$ be a Lipschitz extension of $\pi\circ\varphi$. Then for almost all $y \in \R^n$ and for every $z \in K \cap \psi^{-1}(y)$ the local degrees $\deg(\overline{\varphi}, z), \deg(\pi, \overline{\varphi}(z))$ are well-defined and satisfy
            \begin{equation}\label{eq:degree-llc-extension}
                \deg(\overline{\varphi}, z)\cdot \deg(\pi, \overline{\varphi}(z)) = \sgn(\det D_z\psi).
            \end{equation}
        \end{lemma}

        Notice that \eqref{eq:degree-llc-extension} implies that, for almost all $y\in \R^n$ and for every $x \in \pi^{-1}(y) \cap \varphi(K)$, the local degree $\deg(\pi, x)$ is equal to $1$ or $-1$. Moreover, taking for $\pi$ a Lipschitz extension of $\varphi^{-1}$ and $\psi = \id_{\R^n}$, we conclude that the local degree $\deg(\overline{\varphi}, z)$ is equal to $1$ or $-1$ for almost all $z \in K$.
        
        \begin{proof}
            Let $C$ be the constant from \eqref{eq:bi-lip-ext-LLC} and let $L\geq 1$ be the bi-Lipschitz constant of $\varphi$. For the moment, fix a density point $z \in K$ such that $\psi$ is differentiable at $z$. Then there exists $r>0$ such that whenever $u \in B(z,r)$, then there exists $v \in K$ satisfying $\norm{u-v} < (2CL)^{-1} \norm{u-z}$. For such $u$ and $v$ we have
            $$d(\varphi(z),\overline{\varphi}(u)) \geq L^{-1}  \norm{z-v} - C\norm{u-v} \geq (2L)^{-1} \norm{u-z}.$$
            Therefore, $\overline{\varphi}(u) \neq \varphi(z)$ for all $u \in B(z,r)\setminus z$ and in particular, the local degree $\deg(\overline{\varphi}, z)$ is well-defined. Furthermore, for $\epsilon>0$ there exists $r>0$ such that for every $u\in B(z,r)$ the following holds:
            \begin{enumerate}
                \item $\norm{\psi(z)-\psi(u)-D_z\psi(u-z)}\leq  2^{-1} \epsilon \norm{u-z}$;
                \item there exists $v\in K$ satisfying $\norm{u-v}< (2(\Lip(\psi)+C)^{-1} \epsilon\norm{u-z}$.
            \end{enumerate}
            Since $\psi$ and $f = \pi\circ\overline{\varphi}$ agree on $K$ we have

            \begin{align*}
                 \norm{f(z)-&f(u)-D_z\psi(u-z)} \leq \norm{\psi(z)-\psi(u)-D_z\psi(u-z) }+\norm{\psi(u)-f(u)}
                 \\
                 &\leq (2\epsilon)^{-1}\norm{u-z}+ \norm{\psi(u)-\psi(v)}+ \norm{f(v)-f(u)} \leq \epsilon\norm{u-z}
            \end{align*}


            for all $u \in B(z,r)$. It follows that $\pi\circ\overline{\varphi}$ is differentiable at $z$ with $D_z(\pi\circ\overline{\varphi}) = D_z\psi$. The area and coarea formula imply that for almost all $y \in \R^n$ the preimage $\psi^{-1}(y) \cap K$ is finite and $\psi$ is differentiable at each $z \in \psi^{-1}(y) \cap K$ with non-degenerate derivative. Notice that for such $z$ the local degrees $\deg(\psi,z)$ and $\deg(\pi,\overline{\varphi}(z))$ are well-defined. Since almost every $z \in K$ is a density point, the degree implies
            $$\sgn(\det D_z \psi) = \sgn(\det D_z (\pi\circ \overline{\varphi})) = \deg (\pi\circ \overline{\varphi},  z) = \deg(\overline{\varphi},z) \cdot \deg(\pi,\overline{\varphi}(z))$$
            for almost all $y \in \R^n$ and for every $z \in K \cap \psi^{-1}(y)$. This completes the proof. 
        \end{proof}

        The local degree in the previous lemma is defined through singular homology groups with integer coefficients. However, the statement remains true if the degree is defined using the singular homology groups with coefficients in $\Z/2\Z$, denoted by $\deg(\varphi,x,\Z/2\Z)$, as explained in Section \ref{sec: orientation-degree}. More precisely, Lemma \ref{lemma: bi-lip-ext-llc-degree} states that for almost all $y\in \R^n$ and for every $x \in \pi^{-1}(y) \cap \varphi(K)$ we have $\deg(\pi,x,\Z/2\Z) = [1]\in \Z/2\Z$. We obtain the following corollary, which is a well-known result if $X$ is smooth.

        \begin{corollary}\label{cor: number-preimage-even}
            Let $X$ be a closed metric $n$-manifold with finite Hausdorff $n$-measure and let $\pi \colon X \to \R^n$ be Lipschitz. If $X$ is almost everywhere linearly locally contractible and $n$-rectifiable, then $\#\pi^{-1}(y)$ is even for almost all $y \in \R^n$.
        \end{corollary}

        Observe that, once we proved Corollary \ref{cor: metric-mfld-rect}, the assumption that $X$ is $n$-rectifiable is redundant.

        \begin{proof}
            Since $X$ is $n$-rectifiable, there exist countably many compact sets $K_i \subset \R^n$ and bi-Lipschitz maps $\varphi_i\colon K_i \to X$ such that the $\varphi_i(K_i)$ are pairwise disjoint and 
            $$\Ha^n\left( X\setminus \bigcup_i \varphi_i(K_i)\right) = 0.$$
            Applying Lemma \ref{lemma: bi-lip-ext-llc-degree} to each $\varphi_i$ implies that for almost all $y \in \R^n$ and for every $x \in \pi^{-1}(y)$, the preimage $\pi^{-1}(y)$ is finite and the local degree $\deg(\pi,x,\Z/2\Z)$ is non-zero. Therefore, by the additivity formula of the degree
            $$\deg(\pi,\Z/2\Z)  = \sum_{x \in \pi^{-1}(y)} \deg(\pi,x,\Z/2\Z) = \big[\#\pi^{-1}(y)\big]\in \Z/2\Z$$
            for almost all $y \in \R^n$. Since $X$ is closed we have $\deg(\pi,\Z/2\Z) = [0] \in \Z/2\Z$; see e.g. \cite[Page 269]{dold-1980}. It follows that $\#\pi^{-1}(y)$ is even for almost all $y\in \R^n$.
        \end{proof}

        An analogous argument shows that whenever $\pi \colon X \to M$ is Lipschitz and $MX$ is a closed Riemannian manifold, then $\#\pi^{-1}(y)$ is even in case $\deg(\pi,\Z/2\Z)$ is zero and odd otherwise.

\section{Orientable manifolds}\label{sec: orientable-mflds}

    The goal of this section is to prove Theorem \ref{thme: existence-orient-llc}. The proof is split into two parts. We first show that there exists a metric fundamental cycle satisfying the upper bound in \eqref{eq: mass-measure-hausdorff-existence-weak-llc}.

    \begin{lemma}\label{lemma: existence-weak-llc}
        Let $X$ be a closed, oriented metric $n$-manifold with finite Hausdorff $n$-measure. Suppose that $X$ is almost everywhere linearly locally contractible. Then $X$ has a metric fundamental class $T \in \bI_n(X)$ with a parametrization $\{\varphi_i,K_i,\theta_i\}$ satisfying $|\theta_i| = 1$ for almost every $x \in K_i$ and each $i \in \N$.
    \end{lemma}

    \begin{proof}
         Let \(X=E\cup P\) be a partition where \(E\) is \(n\)-rectifiable and \(P\) is purely \(n\)-unrectifiable. There exists a countable collection of compact set \(K_i\subset \R^n\) and bi-Lipschitz maps \(\varphi_i \colon K_i \to Y\) such that the images $\varphi_i(K_i)$ are pairwise disjoint and 
        \[
            \Ha^n\Big(E \setminus \bigcup_{i\in \N} \varphi_i(K_i)\Big)=0.
        \]
        By partitioning each $K_i$ if necessary, we may assume that the image of each $\varphi_i$ is fully contained in a linearly locally contractible subset of $X$. For each $i \in \N$, let $\overline{\varphi}_i \colon U_i \to C$ be a continuous extension of $\varphi_i$ given by Lemma \ref{lemma:bi-lip-ext-LLC}. It follows from Lemma \ref{lemma: bi-lip-ext-llc-degree} that the local degree $\deg(\overline{\varphi}_i,z)$ exists and is equal to $1$ or $-1$ for almost all $z\in K_i$ and each $i \in \N$. We define
            $$T=\sum_{i\in \N} \varphi_{i\#} \bb{\theta_i}$$
        where $\theta_i = \deg(\overline{\varphi}_i,z)$. Now, let $\pi \colon X \to \R^n$ be a Lipschitz map and suppose that $\Ha^n(\pi(P)) = 0$. Fix $i \in \N$ for the moment.  Let $\psi \colon \R^n \to \R^n$ be a Lipschitz extension of $\pi\circ \varphi_i$. Lemma \ref{lemma: bi-lip-ext-llc-degree} implies that for almost all $y \in \R^n$ and for every $z \in K \cap \psi^{-1}(y)$, the local degrees $\deg(\overline{\varphi}_i, z), \deg(\pi, \overline{\varphi}(z))$ are well-defined and satisfy
            $$\deg(\overline{\varphi}, z)\cdot \deg(\pi, \overline{\varphi}(z)) = \sgn(\det D_z\psi).$$
        Therefore, 
        \begin{align*}
             \varphi_{i\#} \bb{\theta_i}(1,\pi) &= \int_{K_i}\theta_i(z) \det(D_z\psi) \; d\leb^n(z)
             \\
             &=  \int_{K_i} \deg(\overline{\varphi}_i,z) \sgn(\det D_z\psi) |\det(D_z\psi)| \; d\leb^n(z)
             \\
             &= \int_{K_i} \deg(\pi, \overline{\varphi}(z))|\det(D_z\psi)| \; d\leb^n(z).
        \end{align*}

        By the area formula and because $\varphi_i$ is bi-Lipschitz we have

        $$ \varphi_{i\#} \bb{\theta_i}(1,\pi) = \int_{\R^n}\left(\sum_{x\in \pi^{-1}(y)}\mathbbm{1}_{\varphi_i(K_i)}(x)\deg(\pi,x)\right)\,d\leb^n(y).$$
        Recall that $\Ha^n(\pi(P))=0$ and hence, by the additivity property of the degree 
        \begin{align*}
            T(1,\pi) &= \sum_{i\in\N}  \varphi_{i\#} \bb{\theta_i}(1,\pi)
            \\
            &=  \int_{\R^n}\left(\sum_{x\in \pi^{-1}(y)}\deg(\pi,x)\right)\,d\leb^n(y)
            \\
            &= \int_{\R^n} \deg(\pi) \,d\leb^n(y).
        \end{align*}
        Since $X$ is closed we have $\deg(\pi)=0$ and in particular, $T(1,\pi)= 0$ for every Lipschitz map $\pi\colon X\to \R^n$ satisfying $\Ha^n(\pi(P))= 0$. If $\pi \colon X \to \R^n$ is any Lipschitz map, then Theorem \ref{thme: bate-perturb} implies that there exists a sequence of Lipschitz maps $\pi_i \colon X \to \R^n$ converging uniformly to $\pi$ and satisfying $\Lip\pi_i = \Lip\pi$ as well as $\Ha^n(\pi_i(P)) = 0$ for each $i \in \N$. It follows from the continuity property of metric currents that $T(1,\pi) = \lim_{i \to \infty}T(1,\pi_i) = 0$ and thus $\partial T= 0$. It remains to show that $\psi_\#T= \deg(\psi) \cdot \bb{M}$ for every Lipschitz map $\psi\colon X \to M$. Again, we first suppose that $\Ha^n(\psi(P)) = 0$. Then, a computation as above using the area formula and Lemma \ref{lemma: bi-lip-ext-llc-degree} shows
        $$
        \psi_\#\varphi_{i\#}\bb{\theta_i}(g,\tau) = \int_{\R^n}\left(\sum_{x\in \psi^{-1}(\tau^{-1}(y))}\mathbbm{1}_{\varphi_i(K_i)}(x)g(\psi(x))\deg(\tau\circ \psi, x)\right)\,d\leb^n(y)
        $$
        for every $i \in \N$ and all $(g,\tau) \in \mathcal{D}^n(M)$. Furthermore, it follows from  the area formula and the coarea formula that for almost every $y\in\R^n$ the preimage $\tau^{-1}(y)$ is finite, does not intersect $\psi(P)$, and for every $z\in \tau^{-1}(y)$ we have that $\psi^{-1}(z)$ is finite and $\tau$ is differentiable at $z$ with non-degenerate derivative. Therefore, by the multiplicity property of the local degree we have 
        $$
        \deg(\tau\circ \psi, x)= \deg(\tau, z)\cdot \deg(\psi, x) = \sgn(\det D_z\tau)\cdot \deg(\psi, x),
        $$
        for every such $y$ and each $z\in\tau^{-1}(y)$ and $x\in \psi^{-1}(z)$. Since $\Ha^n(\psi(P))=0$ and by the additivity property of the degree we get
        \begin{equation*}
        \begin{split}
             \psi_\#T(g, \tau) &=\int_{\R^n}\left(\sum_{z\in\tau^{-1}(y)}\left(\sum_{x\in \psi^{-1}(z)} \deg(\psi, x)\right)\cdot g(z)\sgn(\det D_z\tau)\right)\,d\leb^n(y)
             \\
              &= \deg(\psi ) \int_{\R^n}\left(\sum_{z\in\tau^{-1}(y)} g(z)\sgn(\det D_z\tau)\right)\,d\leb^n(y).
        \end{split}
        \end{equation*}
        The change of variables formula implies that $\psi_\#T= \deg(\psi)\cdot \bb{M}$. Finally, let $\psi \colon X \to M$ be any Lipschitz map. Since $M$ is a closed Riemannian $n$-manifold, there exists a bi-Lipschitz embedding $\iota$ of $M$ into Euclidean space and a Lipschitz retraction $\eta$ of an open neighborhood of $\iota(M)$. We can apply Theorem \ref{thme: bate-perturb} to the composition $\iota\circ \psi$ to obtain a $\Lip(\iota\circ \psi)$-Lipschitz map $\psi'\colon X\to M$ arbitrarily close to $\psi$ and satisfying $\Ha^n(\psi'(P))= 0$. Using the Lipschitz retraction $\eta$, we conclude that there exists a sequence $\psi_i \colon X \to M$ converging uniformly to $\psi$ such that $\Lip(\psi_i)\leq \Lip(\eta\circ \iota\circ\psi)$ and $\Ha^n(\psi_i(P)) = 0$ for each $i \in \N$. It follows from Lemma \ref{lemma: lip-retract-homotopies} that for $i \in \N$ large enough we have $\deg(\psi) = \deg(\psi_i)$. Therefore, by the continuity property of metric currents
        $$\psi_{\#}T(g,\tau) = \lim_{i \to \infty} \psi_{i\#} T(g,\tau) = \lim_{i \to \infty}    \deg(\psi_i) \cdot\bb{M}(g,\tau) = \deg(\psi) \cdot\bb{M}(g,\tau).$$
        for all $(g,\tau) \in \mathcal{D}^n(M)$. This completes the proof.
    \end{proof}

    The lower bound in \eqref{eq: mass-measure-hausdorff-existence-weak-llc} will be a direct consequence of the following lemma.

    \begin{lemma}\label{lemma: lower-bound-current}
        Let $X$ be a closed, oriented metric $n$-manifold with finite Hausdorff $n$-measure. Suppose that $X$ has a metric fundamental class $T\in \bI_n(X)$. If $X$ is $(\lambda,R)$-linearly locally contractible around $x$, then
        \begin{equation}\label{eq:lower-bound-3}
            \norm{T}(B(x,r)) \geq c \cdot r^n,
        \end{equation}
        for every $r \in (0, R)$. Here, $c$ depends only on $\lambda$ and $n$. 
    \end{lemma}

    The argument essentially follows the proof of \cite[Proposition 5.7]{BMW}. 

    \begin{proof}
        By the slicing inequality
        $$ \int_0^r\mass(\partial(T\mres B(x,s)))\,ds \leq \norm{T}(B(x,r)),$$ 
        for all $r>0$. Therefore, it suffices to prove that there exists a constant $c$, depending only on $\lambda$ and $n$, such that
        \begin{equation}\label{eq:lower-bound-slices}
            \mass(\partial(T\mres B(y,r)))\geq c r^{n-1},
        \end{equation}
         for almost every $r\in \big(0, R/2\big)$. Let $r\in \big(0, R/2\big)$ be such that $T\mres B(x,r)\in\bI_n(X)$ and $\partial T\mres B(x,r)$ is supported in $\{y\in X: d(x, y)=r\}$. Suppose that \eqref{eq:lower-bound-slices} is not true for $c= n(4(5Q+3)D^{1+\frac{1}{n}})^{-(n-1)}$. Here, $D\geq1$ is the constant from Theorem \ref{thme: filling-inj-spaces} for $k = n-1$ and $Q\geq 1$ is the constant from Proposition \ref{prop:eps-cont-extension} for $\lambda$ and $n$. Embed $X$ into $E(X)$ and let $V\in \bI_n(E(X))$ be a minimal filling of $U= \partial(T\mres B(x,r))$ given by Theorem \ref{thme: filling-inj-spaces}. Therefore,
         $$\mass(V) \leq D \mass(U)^{\frac{n}{n-1}} <Dcr^n \quad \text{ and } \quad \norm{V}(B(y, r))\geq D^{-n}r^{n}$$
         for each \(y\in \spt V\) and all \(r \in (0, d(y, \spt U))\). It follows that
         \begin{equation}\label{eq:control-support-V}
              d_H(\spt V, \spt U)\leq c'\cdot r,
         \end{equation}
        where \(d_H\) denotes the Hausdorff distance and $c' = D^{1+\frac{1}{n}}\cdot (c n)^{\frac{1}{n-1}}$. Furthermore, we have $\textup{set}V = \spt V$. We conclude that the Hausdorff dimension of $\spt V$ is bounded by $n$ and thus, its topological dimension is bounded by $n$ as well. Let $Y = X \cup \spt V$ and let $\pi_0\colon Y \to X$ be a map such that $\pi_0(x) =x$ for all $x \in X$ and $d(\pi_0(y), y) \leq 2 d(X,y)$ for all $y \in \spt V$. Using \eqref{eq:control-support-V}, it is not difficult to show that $\pi_0$ is $(5c'r)$-continuous. Recall that $B(x,R)$ is $(\lambda,R)$-linearly locally contractible. Since $5c'r< Q^{-1}R$, Proposition \ref{prop:eps-cont-extension} implies that there exists a continuous map $\pi \colon Y \to X$ satisfying $\pi(x) = x$ for all $x \in X$ and $d(\pi(y),\pi_0(y)) \leq 5Qc'r$ for all $y \in \spt V$. Therefore,
        $$d(\pi(y) ,y) \leq d(\pi(y),\pi_0(y))+d(\pi_0(y),y)) \leq (5Q+3)c'r$$
        for each $y \in \spt V$. Thus, by our choice of $c$, we have 
        \begin{equation}\label{eq:property-retraction-eps-cont}
            d_H(\pi(\spt V), \spt U)\leq \frac{r}{4}.
        \end{equation}
        Let $\varrho\colon X \to M$ be a homeomorphism of degree one, where $M$ is a closed, oriented, smooth $n$-manifold. Using that $\varrho$ is a homeomorphism and $X$ is compact, we conclude that there exists $\epsilon\in (0,r/2)$ such that
        \begin{equation}\label{eq: llc-lower-bound-1}
            N_\epsilon(\varrho(A) )\subset \varrho\big(N_{\frac{r}{4}}(A)\big)
        \end{equation}
        for all $A \subset X$. Furthermore, by \cite[Lemma 5.8]{BMW}, and by decreasing $\epsilon$ if necessary, every Lipschitz map $\varphi \colon Y \to M$ with $d(\varphi,\varrho\circ \pi)< \epsilon$ satisfies
        \begin{equation}\label{eq: llc-lower-bound-2}
            \varphi(X \setminus B(x,r)) \subset \varrho(X \setminus B(x,r/2)).
        \end{equation}
        Lemma \ref{lemma: lip-approx-easy} and Lemma \ref{lemma: lip-retract-homotopies} imply that there exists $\varphi \colon Y \to M$ Lipschitz with $d(\varphi,\varrho\circ \pi)< \epsilon$ and $\deg(\varphi) =1$. Set $ W = \varphi_\#(T\on B(x,r)-V) \in \bI_n(M)$ and notice that $\partial W = 0$. It follows from \eqref{eq:property-retraction-eps-cont} together with \eqref{eq: llc-lower-bound-1} that
        \begin{align*}
            \spt (\varphi_\#V) &\subset N_\epsilon(\varrho(\pi(\spt V)))  \subset \varrho\big(N_\frac{r}{4}(\pi(\spt V))\big)
            \\
            &\subset \varrho\left(\left\{y\in X: \frac{r}{2}<d(x,y)<\frac{3r}{2}\right\}\right).
        \end{align*}
        Therefore, $\spt W \subset \varrho(B(x,r)) \neq M$ and the constancy theorem in $M$ ,\cite[Corollary~3.13]{federer-1960}, implies that $W = 0$. Furthermore,
        $$ \varphi_\#(T\on B(x,r)) = \bb{M} - \varphi_\#(T\on (X\setminus B(x,r)))$$ 
        and, hence by \eqref{eq: llc-lower-bound-2}
        $$ \spt\big(\varphi_\#(T\on (X\setminus B(x,r))\big) \subset \varphi(X\setminus B(x,r))\subset \varrho(X\setminus B(x,r/2)).$$ 
        We conclude that $\varrho(B(x,r/2))\subset \spt\big(\varphi_\#(T\mres B(x,r))\big)$ and thus, $W \neq 0$. This yields a contradiction and completes the proof.
    \end{proof}

    We can now prove the existence of a metric fundamental class for compact, oriented metric manifolds.
    
    \begin{proof}[Proof of Theorem \ref{thme: existence-orient-llc}]
        Let $X$ be a metric space with Hausdorff $n$-measure that is almost everywhere linearly locally contractible and homeomorphic to a compact, oriented smooth $n$-manifold. We first prove the theorem under the additional assumption that $X$ has no boundary. Let $T\in \bI_n(X)$ be the integral $n$-cycle given by Lemma \ref{lemma: existence-weak-llc}. Since $X$ is almost everywhere linearly locally contractible, Lemma \ref{lemma: lower-bound-current} implies that
        $$\liminf_{r\to 0}\frac{\norm{T}(B(x,r))}{r^n}>0$$
        for almost all $x \in X$. Therefore, the characteristic set of $T$ is equal to $X$. Furthermore, $T$ is given by a parametrization $\{\varphi_i,K_i,\theta_i\}$ with $|\theta_i(x)|=1$ for almost all $x \in K_i$ and every $i \in \N$. It follows that $\norm{T} = \lambda \Ha^n$, where $\lambda\in L^1(X)$ satisfies $n^{-n/2}\leq \lambda \leq n^{n/2}$. Finally, \cite[Proposition 5.5]{BMW} implies that $T$ is a generator of $H_n^\textup{IC}(X)$. This completes the proof in case $X$ has no boundary.
        \medskip

        Now, suppose that $X$ has boundary and $\Ha^n(\partial X ) = 0$. Let $\hat{X}$ be the closed manifold obtained by gluing two copies of $X$ along their boundaries and equip $\hat{X}$ with the quotient metric; see Definition \ref{def: gluing-boundary}. It follows from Lemma \ref{lemma: ae-llc-bdry-double-ae-llc} that $\hat{X}$ is almost everywhere linearly locally contractible. Therefore, by the first part of the proof $\hat{X}$ has a metric fundamental class $\hat{T}\in \bI_n(\hat{X})$ generating $H_n^\textup{IC}(\hat{X})$ and satisfying
        $$n^{-n/2}\Ha_{\hat{X}}^n \leq \norm{\hat{T}} \leq n^{n/2}\Ha_{\hat{X}}^n.$$
        Let $\iota\colon X \to \hat{X}$ an isometric embedding that identifies $X$ with one of its copies in $\hat{X}$. By Lemma \ref{lemma: fund-class-closed-gives-fund-class-bdry} the $n$-current $T = \iota^{-1}_{\#}(\hat{T} \on \iota(X)) \in \cI_n(X)$ defines a metric fundamental class of $X$ satisfying 
        $$n^{-n/2}\Ha_{X}^n \leq \norm{T} \leq n^{n/2}\Ha_{X}^n$$
        and whenever $S \in \bI_n(X)$ is such that $\spt S \subset \partial X$, then there exists $k\in \Z$ with $k\cdot T = S$. This completes the proof
    \end{proof}

    We conclude the section with the proofs of the rectifiability result for metric manifolds. 

    \begin{proof}[Proof of Corollary \ref{cor: metric-mfld-rect}]
        As mentioned before, we only assume that $X$ is almost everywhere linearly contractible. Let $\hat{X}$ be the manifold obtained by gluing two copies of $X$ along their boundary. It follows from Lemma \ref{lemma: ae-llc-bdry-double-ae-llc} that $\hat{X}$ is almost everywhere linearly contractible as well. Since $X$ embeds isometrically into $\hat{X}$, it suffices to prove the theorem for closed manifolds. If $X$ is orientable, it follows from Theorem \ref{thme: existence-orient-llc} that there exists $T \in \bI_n(X)$ with $\textup{set} T = X$ and thus, $X$ is $n$-rectifiable. Finally, if $X$ is non-orientable, then Lemma \ref{lemma: ae-llc-orientable-double-ae-lcc} implies that the orientable double cover $\tilde{X}$ is almost everywhere linearly locally contractible. By the above $\tilde{X}$ is $n$-rectifiable. We conclude that $X$ is $n$-rectifiable as well because $\tilde{X}$ locally isometric to $X$.
    \end{proof}

\section{Non-orientable manifolds}\label{sec: non-orient-mflds}
    In this section we provide the proofs for the different existence results for the metric fundamental class modulo $2$.

\subsection{Linear locally contractibility}
    We begin with the following lemma.

    \begin{lemma}\label{lemma: mod-2-pushfwd-0-cycle}
        Let $X$ be a compact metric space and let $T \in \bI_n(X)$ be an integral $n$-current satisfying $\varphi_\# T = 0 \mod 2$ for all Lipschitz maps $\varphi\colon X \to \R^n$. Then, $\partial T = 0 \mod 2$.
    \end{lemma}

    \begin{proof}
        We first prove the result under the additional assumption that $X$ is a subset of a finite dimensional vector space $V$. Clearly, the statement is unaffected by bi-Lipschitz changes of the metric on $V$. Therefore, we may assume that $V$ is equal to some $\R^N$ equipped with the standard Euclidean distance. Let $E\subset X$ be the $(n-1)$-rectifiable set that $\norm{\partial T}$ is concentrated on. By Theorem \ref{thme: bate-jakub} there exists a $1$-Lipschitz map $\varphi \colon \R^N \to \R^n$ such that $\Ha^{n-1}(\varphi(E)) = \Ha^{n-1}(E)$. It follows that $\varphi$ is mass-preserving, that is, $\Ha^{n-1}(\varphi(B)) = \Ha^{n-1}(B)$ for every Borel set $B\subset E$; see e.g. \cite[Lemma 3.1]{marti2025lipschitz}. Combining this with \cite[Lemma 7.2]{bate-perturb}, we conclude that there exist countably many pairwise disjoint subsets $E_i \subset E$ covering $E$ up to a set of $\Ha^{n-1}$-measure zero and such that $\varphi\colon E_i \to \R^n$ is bi-Lipschitz for each $i\in \N$ and the $\varphi(E_i)$ are pairwise disjoint as well. Hence, 
        $$\partial T = \sum_i (\partial T)\on E_i=\sum_i  \varphi_\#^{-1}(\varphi_\#(\partial T)\on E_i ).$$
        Since $\varphi_\#\partial T = 0 \mod 2$ and the $\varphi(E_i)$ are pairwise disjoint, we have $\varphi_\#(\partial T)\on E_i  = 0 \mod 2$ for each $i \in \N$ and in particular, $\partial T =  0 \mod 2$. We now prove the general statement. Embed $X$ into $l^\infty$. It follows from Lemma \ref{lemma: map-modulo-2} that there exist countably many finite dimensional linear subspaces $V_i \subset l^\infty$ and $1$-Lipschitz maps $\pi_i \colon X \to V_i$ such that $\cF_2(\partial T - \pi_{i\#} \partial T) \to 0$ as $i \to \infty$. Fix $i \in \N$ for the moment. If $\varphi\colon V_i \to \R^n$ is Lipschitz, then there exists a Lipschitz extension $\tilde{\varphi}\colon l^\infty \to \R^n$ of $\varphi$ and we have $\varphi_\#(\pi_{i\#}T) = (\tilde{\varphi}\circ \pi_i)_\# T= 0 \mod 2$. Therefore, $\varphi_\#(\pi_{i\#}T)=0\mod 2$ for every Lipschitz map $\varphi\colon V_i \to \R^n$. By the first part of the proof this implies $\partial (\pi_{i\#}T) = 0 \mod 2$. We conclude that $\partial T = 0 \mod 2$.
    \end{proof}
    
    Next, we prove the existence of a metric fundamental class modulo $2$ in a closed metric manifold.

    \begin{proposition}\label{prop: existence-fund-class-nonorientable}
        Let $X$ be a closed, non-orientable metric $n$-manifold that is almost everywhere linearly locally contractible. Then, $X$ has a metric fundamental class $T\in \bI_n(X)$ modulo $2$ satisfying
        $$C^{-1} \Ha^n \leq \norm{T}_2= \norm{T}\leq C \Ha^n.$$
        Here, $C>0$ depends only on $n$.
    \end{proposition}
    
    \begin{proof}
        Let $\tilde{X}$ be the orientable double cover of $X$ and let $\pi \colon \tilde{X} \to X$ be the covering map. It follows from Lemma \ref{lemma: ae-llc-orientable-double-ae-lcc} that $\tilde{X}$ is almost everywhere linearly locally contractible. Hence, Theorem \ref{thme: existence-orient-llc} implies that $\tilde{X}$ has a metric fundamental class $\tilde{T}\in \bI_n(\tilde{X})$ satisfying
        \begin{equation}\label{eq: mass-measure-orientable-double-cover}
            C^{-1}\cdot \Ha_{\tilde{X}}^n\leq \norm{\tilde{T}}\leq C\cdot \Ha_{\tilde{X}}^n,
        \end{equation}
        where $C>0$ depends only on $n$. There exists $R>0$ such that for each $x\in X$ and all $0<r<R$, the preimage $\pi^{-1}(B(x,r))$ consists of two components and $\pi$ restricted either component is an isometry. Let $x_1,\dots,x_N\in X$ and $r_1,\dots,r_N \in (0,R)$ be a collection of points and radii with the following properties. The balls $B_i = B(x_i,r_i)$ cover $X$ and for each $i = 1,\dots, N$, the slice $\langle \tilde{T}, \pi_i,r_i\rangle$ is an integral $(n-1)$-cycle, where $\pi_i \colon \tilde{X}\to \R, \;x\mapsto d(x_i,\pi(x))$. Notice that for each $i \in \N$,
        $$\partial(\tilde{T}\on \pi^{-1}(B_i) )=\partial( \tilde{T}\on B(x_i^+,r_i)) +\partial(\tilde{T}\on B(x_i^-,r_i) ) = \langle \tilde{T}, \pi_i,r_i\rangle,$$
        where $\pi^{-1}(x_i)= \{x_i^+,x_i^-\}$ and hence, $\tilde{T}\on \pi^{-1}(B_i) \in \bI_n(\tilde{X})$. Let $U_1 = B_1$ and for $i=2,\dots,N$ set $U_i = B_i \setminus \bigcup_{j=1}^{i-1}U_j$. It follows that $\{U_i\}_{i=1}^N$ is a partition of $X$ and for each $i = 1,\dots, N$ the preimage $\pi^{-1}(U_i)$ consists of two components $U_i^+,U_i^-\subset \tilde{X}$ such that $\tilde{T}\on U_i^+,\tilde{T}\on U_i^-\in \bI_n(\tilde{X})$. For each $i = 1,\dots, N$ choose one component $\tilde{U}_i$ of $\pi^{-1}(U_i)$ and define 
        $$T = \sum_i^N \pi_\#(\tilde{T}\on \tilde{U}_i) \in \bI_n(X).$$
        Since the $\tilde{U}_i$ are pairwise disjoint and $\pi$ is an isometry on each $\tilde{U}_i$ it follows from \eqref{eq: mass-measure-orientable-double-cover} that
        $$C^{-1}\Ha^n_X \leq \norm{T} \leq C \Ha^n_X.$$
        Moreover, by Theorem \ref{thme: existence-orient-llc} there exists a parametrization $\{\varphi_j,K_j,\theta_j\}$ of $\tilde{T}$ such that $|\theta_j(x)| =1$ for almost all $x \in K_j$ and every $j \in \N$. We conclude that the integral $n$-current $T$ is already a reduction modulo $2$ and hence, Lemma \ref{lemma: mass-equal-mod-mass-for-reduction} implies that $\norm{T}_2 = \norm{T}$. Next, we show that $\partial T = 0 \mod 2$. Let $f \colon X \to \R^n$ be Lipschitz. By Corollary \ref{cor: number-preimage-even} and Corollary \ref{cor: metric-mfld-rect} we have that $\# f^{-1}(z)$ is even for almost all $z \in \R^n$. Therefore, there exists a parametrization $\{(\varphi_{k,l},\varphi_N),(K_{k,l},N),(\theta_{k,l},\theta_N)\}$ of $T$ such that $|\theta_{k,l}(y)|=1$ for almost all $y \in K_{k,l}$ and every $k,l\in \N$ and $\Ha^n(f(N))=0$ as well as
        $$\bigcup_l \varphi_{k,l}(K_{k,l}) = f^{-1}(\{z \in \R^n \colon \#f^{-1}(z) = 2k\}$$
        for all $k \in \N$. Notice that changing the sign of some $\theta_{k,l}$ on a subset of $K_{k,l}$ induces the same current modulo $2$. Thus, we may assume that for each $k \in \N$ we have $$f_\#\left(\sum_l\varphi_{k,l\#}\bb{\theta_{k,l}}\right) = 2k \cdot \bb{A_k},$$
        where $A_k =\{z \in \R^n \colon \#f^{-1}(z) = 2k\}$ and therefore 
        $$f_\# T = 2S = 2 \left(\sum_{k=1}^\infty k \cdot\bb{A_k}\right).$$
        Since $f$ is Lipschitz, the area formula implies that $S \in \cI_n(\R^n)$ has finite mass and hence, $f_\# T = 0 \mod 2$. Lemma \ref{lemma: mod-2-pushfwd-0-cycle} implies that $\partial T = 0 \mod 2$. Finally, let $M$ be a smooth $n$-manifold homeomorphic to $X$ and let $\psi\colon X \to M$ be Lipschitz. It can be shown analogously as in Corollary \ref{cor: number-preimage-even}, that for almost all $z \in M$ the number of preimages $\#\psi^{-1}(z)$ is even if $\deg(\psi,\Z/2\Z) = [0]$ and odd if $\deg(\psi,\Z/2\Z) = [1]$. The same argument as above, using a parametrization of $T$ adapted to $\psi$, shows that there exists $S \in \cI_n(M)$ with $\psi_\# T = \deg(\psi,\Z/2\Z) \cdot \bb{M} +2S$. In particular, $\psi_\# T = \deg(\psi,\Z/2\Z) \cdot \bb{M} \mod 2$. This completes the proof. 
    \end{proof}

    Finally, we prove the uniqueness of the metric fundamental class modulo $2$ in a non-orientable metric
    manifold. The proof is an adaptation of \cite[Proposition 5.5]{BMW}.

    \begin{lemma}\label{lemma: uniqueness-mod-2}
        Let $X$ be a closed, non-orientable metric $n$-manifold with finite Hausdorff $n$-measure. Suppose that $X$ has a metric fundamental class $T\in \cI_n(X)$ modulo $2$. Then, every $S\in \cI_n(X)$ with $\partial S =0 \mod 2$ satisfies $k T = S\mod 2$ for $k$ either equal to $0$ or $1$. 
    \end{lemma}

    \begin{proof}
        Let $\varrho\colon X \to M$ be a homeomorphism, where $M$ is a closed non-orientable Riemannian manifold. By Lemma \ref{lemma: lip-retract-homotopies} there exists $\delta>0$ such that whenever $\varphi,\psi\colon X\to M$ are Lipschitz maps satisfying $d(\varphi,\psi)<2\delta$ then there exists a Lipschitz homotopy $H\colon [0,1]\times X\to M$ between $\varphi$ and $\psi$. It follows from \eqref{eq: homotopy-boundary} 
        $$\partial H_\#(\bb{0,1}\times S) + H_\#(\bb{0,1}\times\partial S) =  \varphi_\#S- \psi_\#S.$$
        Since $\cI_{n+1}(M) = 0$, we have $\partial H_\#(\bb{0,1}\times S)=0$. Furthermore, $\partial S = 0\mod 2$ and thus, $\varphi_\#S = \psi_\#S\mod 2$. By the uniqueness of $\bb{M}$, we conclude that there exists $k$ equal to either $0$ or $1$ such that $\varphi_\# S = k \cdot\bb{M} \mod 2$ for every Lipschitz map $\varphi\colon X \to M$ with $d(\varphi,\varrho) <\delta$. 
        Embed $X$ into $l^\infty$. We claim that for every $\epsilon>0$ there exists $U \in \cI_{n+1}(l^\infty)$ with $\spt U \subset N_\epsilon(X)$ and $U$ is a filling modulo $2$ of $S-kT$, that is, $\partial U = S- kT \mod 2$. Let $\epsilon>0$. We first apply Lemma \ref{lemma: lip-approx-easy} to $\varrho^{-1}$ and obtain a Lipschitz map $g\colon M \to l^\infty$ with $d(\varrho^{-1},g)<\epsilon/2$. Then we use Lemma \ref{lemma: lip-approx-easy} again to find $f\colon X \to M$ Lipschitz with $d(\varrho,f)<(\epsilon/2\Lip(g))$. It follows that
        \begin{equation*}\label{eq: uniq-mod-2-approxs}
            d(g\circ f,\varrho^{-1}\circ\varrho) \leq d(g\circ f,g\circ \varrho) + d(g\circ \varrho,\varrho^{-1}\circ\varrho) < \epsilon.
        \end{equation*}
        Set $Q =S- kT$. We may suppose that $\deg(f,\Z/2\Z)=1$. Therefore, since $T$ is a metric fundamental class modulo $2$, we have $f_\# Q =  0\mod 2$ and in particular, $(g\circ f)_\# Q = 0\mod 2$. Let $H\colon [0,1]\times X\to l^\infty$ be the straight line homotopy between $(g\circ f)$ and the inclusion $\iota_X$ of $X$ into $l^\infty$. Put $U = H_\#(\bb{0,1}\times Q)\in\cI_{n+1}(l^\infty)$. We have $H(t,x) \in N_\epsilon(X)$ for all $x\in X$ and every $t \in [0,1]$ and hence, $\spt U \subset N_\epsilon(X)$. Moreover, using \eqref{eq: homotopy-boundary} we get
        \begin{align*}
            \partial U 
            &= Q-(g\circ f)_\# Q - H_\#(\bb{0,1}\times\partial Q).
        \end{align*}
        Since $\partial Q = 0 \mod 2$ and $(g\circ f)_\# Q = 0\mod 2$ we conclude that $U$ is a filling modulo $2$ of $Q = kT - S$. This proves the claim. Finally, we show that $k T = S \mod 2$. Lemma \ref{lemma: map-modulo-2} implies that there exist countably many finite dimensional linear subspaces $V_i\subset l^\infty$ and $1$-Lipschitz maps $\pi_i\colon l^\infty\to V_i$ such that $\cF_2(Q -\pi_{i\#}Q) \to 0$ as $i \to 0$. Fix $i \in \N$ for the moment. Let $N = \dim V_i$, $Y = \pi_i(X)$ and $Q' = \pi_{i\#}Q\in \cI_{n}(V_i)$. It follows from the claim that for each $\epsilon>0$ there exists a filling $U$ modulo $2$ of $Q'$ satisfying $\spt U \subset N_\epsilon(Y) \subset V_i$. We apply a refined version of Federer and Fleming’s deformation theorem \cite[Theorem 4.2.$9^\nu$]{federer-gmt} for integer rectifiable currents mod $2$. More precisely, for $\eta>0$ we choose a cubical
        subdivision of $V_i$ into Euclidean cubes of side length $\eta$. Then we use radial projections to push $Q'$ first into the $(N-1)$-dimensional skeleton, then into the $(N-2)$-dimensional skeleton and so on till we arrive at the skeleton of dimension $n$. Since $\Ha^{n+1}(Y) = 0$ we can achieve this with projections that have their projection centers outside of $Y$; c.f. \cite[Chapter 3]{kinneberg} and \cite[Chapter 5]{marti2024characterization}. It follows that the resulting current $P$ is the pushforward of $Q'$ by a Lipschitz map $p$ defined on $Y$ and $\cF_2(Q' -P) \leq C \eta \bM(Q')$, where $C$ depends only on $N$. Because the $n$-skeleton of the cubical subdivision is an absolute Lipschitz neighborhood retract, we can extend $p$ to an open neighborhood of $Y$. Therefore, there exists filling $U$ modulo $2$ of $Q'$ such that the pushforward $p_\# U$ is well-defined. This implies that there exists an integer rectifiable current $U'\in \cI_{n+1}(V_i)$ in the $n$-skeleton with $\partial U' = P \mod 2$. However, every such $U'$ has to be zero and thus, $P = 0 \mod 2$. Since $\eta>0$ was arbitrary, we conclude that $\pi_{i\#} Q = Q' = 0 \mod 2$. Using that $\cF_2(Q -\pi_{i\#}Q) \to 0$ as $i \to 0$, we get $Q= 0\mod 2$ and in particular, $kT = S \mod 2$. This completes the proof. 
        

    \end{proof}

    The proof of Theorem \ref{thme: existence-non-orient-llc} is a direct consequence of the preceding results.

    \begin{proof}[Proof of Theorem \ref{thme: existence-non-orient-llc}]
        Let $X$ be a metric space with finite Hausdorff $n$-measure that is almost everywhere linearly contractible and homeomorphic to a compact, non-orientable smooth $n$-manifold. First, we assume that $X$ is closed. It follows from Proposition \ref{prop: existence-fund-class-nonorientable} that there exists $T \in \bI_n(X)$ with $\partial T = 0 \mod 2$ satisfying 
        \begin{equation}\label{eq: existence-non-orient-llc-mass-ineq}
            C^{-1}\Ha^n \leq \norm{T}=\norm{T}_2 \leq C \Ha^n,
        \end{equation}
        where $C>0$ depends only on $n$. By Lemma \ref{lemma: uniqueness-mod-2} the current $T$ is unique in the following sense. Whenever $S \in \cI_n(X)$ satisfies $\partial S = 0 \mod 2$, then $k T  = S \mod 2$ for $k$ equal to either $0$ or $1$. Now, suppose that $X$ has boundary and satisfies $\Ha^n(\partial X) = 0$. Let $\hat{X}$ be the manifold obtained by gluing two copies of $X$ along the boundary. Lemma \ref{lemma: ae-llc-bdry-double-ae-llc} implies that $\hat{X}$ is almost everywhere linearly contractible as well. Let $\hat{T}\in \bI_n(\hat{X})$ be the metric fundamental class modulo $2$ of $\hat{X}$ obtained in the first part of the proof. It follows from Lemma \ref{lemma: fund-class-closed-gives-fund-class-bdry} that $(\iota^{-1})_\#\hat{T}\on\iota(X)$ defines a unique metric fundamental class modulo $2$ in $X$. Here, $\iota \colon X\to \hat{X}$ is the embedding that identifies $X$ with one of its copies in $\hat{X}$. Since $\iota$ is an isometric embedding, $T$ satisfies \eqref{eq: existence-non-orient-llc-mass-ineq} as well. This concludes the proof. 
    \end{proof}

\subsection{Nagata dimension}  
    Throughout this section, let $X$ denote a metric space with finite Hausdorff $n$-measure and Nagata dimension less than $N$. Furthermore, let $\varrho\colon M \to X$ be a homeomorphism, where $M$ is a compact, non-orientable Riemannian $n$-manifold. Embed $X$ into its injective hull $E(X)$.

    \begin{lemma}\label{lemma: existence-non-orient-nagata-sequence}
        Suppose that $M$ is closed. Then, for every $\epsilon>0$ there exists a Lipschitz map $\eta\colon M \to E(X)$ such that $d(\varrho,\eta) \leq \epsilon$ and $\bM_2(\eta_\# \bb{M}) \leq C \Ha^n(X)$. Here, $C$ depends only on the data of $X$.
    \end{lemma}

    \begin{proof}
        Let $\epsilon>0$. By Theorem \ref{thme: nagata-factorization} there exists a finite, $n$-dimensional simplicial complex $\Sigma$ equipped with the $l^2$ distance $|\cdot|_{l^2}$ and each simplex in $\Sigma$ is a standard Euclidean simplex of side length $\epsilon$. Furthermore, there exist $C>0$, depending only on the data of $X$, and Lipschitz maps $\varphi\colon X \to \Sigma$, $\psi\colon \Sigma\to E(X)$ with the following properties: $\textup{Hull}(\varphi(X)) = \Sigma$, $\psi$ is $C$-Lipschitz on each simplex, $d(\psi(\varphi(x)),x) \leq C\epsilon$ for every $x \in X$ and $\Ha^n(\varphi(X)) \leq C \Ha^n$. Recall that $\textup{Hull}(\varphi(X))$ is the smallest subcomplex of $\Sigma$ containing $\varphi(X)$ and thus, $\Ha^n(\Sigma) \leq C\Ha^n(X)$. Since $\Sigma$ is an absolute Lipschitz neighborhood retract, Lemma \ref{lemma: lip-approx-easy} implies that there exists a Lipschitz map $\xi\colon M \to \Sigma$ satisfying $d(\xi, \varphi \circ \varrho) < \epsilon$ and $\xi(z), (\varphi \circ \varrho)(z)$ are contained in neighboring simplices for all $z\in M$. Put $\eta = \psi\circ \xi$ and let $z\in M$. It follows from \cite[Lemma 2.3]{BMW} that there exists $y \in \Sigma$ contained in a common simplex with $\xi(z)$ and $(\varphi \circ \varrho)(z)$ and such that 
        $$|\xi(z)-y|_{l^2}+|y-(\varphi \circ \varrho)(z)|_{l^2} \leq 4 \sqrt{n} |\xi(z)-(\varphi \circ \varrho)(z)|_{l^2}.$$
        Therefore, using that $\psi$ is $C$-Lipschitz on each simplex, we have
        \begin{align*}
            d(\eta(z), &\varrho(z)) \leq d(\psi(\xi(z)), \psi(\varphi(\varrho(z)))) + d(\psi(\varphi(\varrho(z))),\varrho(z))
            \\
            &\leq C \big (|\xi(z)-y| + |y-\varphi(\varrho(z))|\big) + C \epsilon \leq (4 \sqrt{n} +1)C \epsilon.
        \end{align*}
        By replacing $\xi_\# \bb{M}$ with a reduction modulo $2$ if necessary, we may assume that $\bM_2(\eta_\#\bb{M}) \leq c_n \Ha^n(\eta(M))$, where $c_n$ depends only on $n$. Therefore, since $\psi$ is $C$-Lipschitz on each simplex of $\Sigma$
        $$\bM_2(\eta_\#\bb{M}) \leq c_n \Ha^n(\eta(M)) \leq c_n \Ha^n(\psi(\Sigma))\leq c_n C^n \Ha^n(\Sigma) \leq c_n C^{n+1}\Ha^n(X).$$
        This completes the proof.        
    \end{proof}

    We can now prove Theorem \ref{thme: existence-non-orient-nagata}.
        
    \begin{proof}[Proof of Theorem \ref{thme: existence-non-orient-nagata}]
        Let $\hat{X}$ be the manifold obtained by gluing two copies of $X$ along their boundaries. It follows that the Nagata dimension of $\hat{X}$ is bounded by $N$ as well. Thus, by Lemma \ref{lemma: fund-class-closed-gives-fund-class-bdry} it suffices to prove the theorem with the additional assumption that $X$ is closed in order to prove the general statement. For $k \in \N$, let $\eta_k \colon M \to E(X)$ be the map given by Lemma \ref{lemma: existence-non-orient-nagata-sequence} for $\epsilon_k= 1/k$ and put $T_k = \eta_{k\#} \bb{M} \in \bI_n(E(X))$. We have 
        $$\sup_i \bM_p(T_i) + \bM_p(\partial T_i) < C \Ha^n(X) < \infty,$$
        where $C$ depends only on the data of $X$. By Theorem \ref{thme: compactness-flat-mod-p} there exist a subsequence of $(T_k)_{k\in \N}$ (still denoted by $(T_k)_{k\in \N}$) and $S \in \cF_n(E(X))$ such that $\cF_2(T_k -S) \to 0$ as $k \to \infty$. Clearly, $S$ is a cycle modulo $2$. Since $d(\varrho,\eta_k) < \epsilon_k$ for each $k\in \N$, the support of each $T_k$ is contained in the $\epsilon_k$ neighborhood of $X$. We conclude that $\norm{S}_p$ is supported on $X$. Therefore, Proposition \ref{prop: rectifiability-flat-current} implies that there exists $T\in \cI_n(X)$ such that $S = T \mod 2$. By Lemma \ref{lemma: mass-equal-mod-mass-for-reduction} and by passing to a reduction of $T$ modulo $2$ if necessary, we may assume that 
        $$\norm{T} = \norm{T}_2 \leq C\Ha^n,$$
        where $C>0$ depends only on $n$. Now, let $\varphi\colon X \to M$ be a Lipschitz map. Since $M$ is an absolute Lipschitz neighborhood retract, we can extend $\varphi$ to some neighborhood of $X$. Hence, for $k \in \N$ sufficiently large the composition $\varphi \circ\eta_k \colon M \to M$ is well-defined and for such $k$ we have 
        $$\varphi_\# T_k = (\varphi \circ \eta_k)_{\#} \bb{M} = \deg (\varphi \circ \eta_k, \Z/2\Z) \cdot \bb{M}.$$
        Since $d(\eta_k,\varrho) < \epsilon_k$ for each $k \in \N$, the composition $\varphi \circ \eta_k$ converges uniformly to $\varphi\circ \varrho $. It follows from Lemma \ref{lemma: lip-retract-homotopies} that for $k$ sufficiently large, the maps $\varphi \circ \eta_k$ and $\varphi\circ \varrho $ are homotopic. Thus, using the homotopy invariance of the degree and that $\varrho$ is a homeomorphism, we get
        $$\varphi_\# T_k =\deg (\varphi \circ \eta_k, \Z/2\Z) \cdot \bb{M} = \deg (\varphi \circ \varrho, \Z/2\Z) \cdot \bb{M} = \deg(\varphi, \Z/2\Z) \cdot \bb{M},$$
        for all $k$ sufficiently large. Passing to the limit we obtain $\varphi_\# T = \deg (\varphi, \Z/2\Z) \cdot \bb{M}$. This completes the proof.
    \end{proof}

\subsection{Surfaces}

    Let $X$ be a metric space with finite Hausdorff $2$-measure that is homeomorphic to a closed smooth surface $M$. It follows from \cite[Theorem 1.3]{nta-rom22} that there exists a Riemannian metric $g$ on $M$ of constant curvature and a continuous, surjective, monotone map $\psi \in W^{1,2}(M, X)$ that is weakly conformal; see also \cite{Meier-Wenger,dimitrios-romney-length}. A map is said to be monotone if it is the uniform limit of homeomorphism from $M$ to $X$. We refer to \cite{nta-rom22} for more details on weakly conformal maps and to \cite{Heinonen-Koskela-Shanmugalingam-Tyson-2015} for more information of the theory of Sobolev maps in metric spaces. One can prove exactly as in \cite[Proposition 6.1]{BMW} that there exist countably many Borel sets $A_k \subset M$ such that $\psi$ restricted to each $A_k$ is $2k$-Lipschitz and $\Ha^n(M\setminus A_k) \leq \epsilon_k/k^2$, where $\epsilon_k$ is some sequence converging to $0$. Using this, the proof of Theorem \ref{thme: existence-non-orient-surface} is not too difficult. 

    \begin{proof}[Proof of Theorem \ref{thme: existence-non-orient-surface}]
        By Lemma \ref{lemma: fund-class-closed-gives-fund-class-bdry} it suffices to prove the theorem when $X$ is homeomorphic to a closed, non-orientable smooth surface $M$. Let $\psi \in W^{1,2}(M, X)$ and $A_k \subset M$, $k \in \N$, be as explained above. Embed $X$ into $E(X)$. Since $\psi$ restricted to each $A_k$ is $2k$-Lipschitz, there exists $2k$-Lipschitz extensions $\psi_k \colon M \to E(X)$ of $\psi|_{A_k}$. For $k \in \N$, we define $T_k = \psi_{k\#}\bb{M} \in \bI_2(E(X))$. We have
        $$\bM_2(T_k) \leq \bM_2(\psi_{k\#}(\bb{M}\on A_k)) + \bM_2(\psi_{k\#}(\bb{M}\on M\setminus A_k))\leq \textup{vol}^*(\psi|_{A_k}) + 4 \varepsilon_k.$$
        for all $k \in \N$. Here, $\textup{vol}^*(\psi|_{A_k})$ denotes the (parametrized) Gromov mass$*$ volume of $\psi$; see \cite{lytchak-wenger1,lytchak-wenger2}. Since $\psi$ belongs to $ W^{1,2}(M, X)$ we have that $\textup{vol}^*(\psi)$ is uniformly bounded. Notice that $\partial T_k = 0 \mod 2$ for all $k\in \N$. Therefore, Theorem \ref{thme: compactness-flat-mod-p} implies that there exists $S \in \cF_2(E(X))$ such that $\cF_2(T_k - S) \to 0$ as $k\to \infty$. It is not difficult to show that the $\psi_k$ converge uniformly to $\psi\colon M \to X$. Hence, $\spt T_k\subset N_{\delta_k}(X)$ for some $\delta_k \to 0$ and in particular, $\spt_2 T \subset X$. It follows from Proposition \ref{prop: rectifiability-flat-current} that there exists $T \in \cI_2(X)$ such that $S = T \mod 2$. By Lemma \ref{lemma: mass-equal-mod-mass-for-reduction} and replacing $T$ by a reduction modulo $2$ if necessary, we may suppose that $\norm{T}_2 = \norm{T} \leq 2 \Ha^2$. Now, let $\varphi\colon X \to M$ be Lipschitz. Recall that $\psi$ is the uniform limit of homeomorphism from $M\to X$ and hence, $\deg(\psi,\Z/2\Z) = [1]\in \Z/2\Z$. We can extend $\varphi$ as a Lipschitz map to an open neighborhood of $X$ in $E(X)$. This is possible because $M$ is an absolute Lipschitz neighborhood retract. It follows that for $k\in \N$ sufficiently large the composition $\varphi\circ \psi_k\colon M \to M$ is well-defined and converges uniformly to $\varphi \circ \psi$. Hence, for $k\in \N$ sufficiently large, Lemma \ref{lemma: lip-retract-homotopies} and the multiplicity property of the degree imply
        $$\varphi_\#T_k = (\varphi\circ \psi_k)_\# \bb{M}= \deg(\varphi\circ\psi_k,\Z/2\Z) \cdot\bb{M} = \deg(\varphi,\Z/2\Z) \cdot\bb{M}.$$
        We conclude that $\varphi_\#T =  \deg(\varphi,\Z/2\Z) \cdot\bb{M}$. Finally, the uniqueness (modulo $2$) of $T$ follows from Lemma \ref{lemma: uniqueness-mod-2}. This completes the proof. 
    \end{proof}

\bibliographystyle{plain}

\end{document}